\documentclass[nohyperref]{article}

\usepackage{microtype}
\usepackage{graphicx}
\usepackage{booktabs} 

\usepackage[nohyperref,accepted]{icml2023}

\usepackage{amsmath}
\usepackage{amssymb}
\usepackage{mathtools}
\usepackage{amsthm}

\usepackage[pagebackref=true]{hyperref}  
\renewcommand*{\backrefalt}[4]{%
    \ifcase #1 \footnotesize{(Not cited.)}%
    \or        \footnotesize{(Cited on page~#2)}%
    \else      \footnotesize{(Cited on pages~#2)}%
    \fi}

\usepackage[capitalize,noabbrev]{cleveref}

\theoremstyle{plain}
\newtheorem{theorem}{Theorem}[section]

\newtheorem{lemma}[theorem]{Lemma}

\theoremstyle{definition}

\theoremstyle{remark}

\icmltitlerunning{Convergence of Proximal Point and Extragradient-Based Methods Beyond Monotonicity: the Case of Negative Comonotonicity}

\usepackage{amssymb, amsmath, amsthm, latexsym}
\usepackage{url}
\usepackage{algorithm}
\usepackage{algorithmic}
\usepackage{tabularx}
\usepackage{paralist}
\usepackage{mathtools}

\usepackage{bbm} 

\usepackage{makecell}
\usepackage{multirow}
\usepackage{booktabs}

\usepackage{nicefrac}       

\usepackage[flushleft]{threeparttable} 

\usepackage{caption}
\usepackage{multirow}
\usepackage{colortbl}
\definecolor{bgcolor}{rgb}{0.8,1,1}
\definecolor{bgcolor2}{rgb}{0.8,1,0.8}
\definecolor{niceblue}{rgb}{0.0,0.19,0.56}

\usepackage{hyperref}
\hypersetup{colorlinks,linkcolor={blue},citecolor={niceblue},urlcolor={blue}}

\renewcommand{\algorithmicrequire}{\textbf{Input:}}
\renewcommand{\algorithmicensure}{\textbf{Output:}}

\usepackage{pifont}
\definecolor{PineGreen}{RGB}{0,110,51}
\definecolor{BrickRed}{RGB}{143,20,2}
\newcommand{\xmark}{{\color{BrickRed}\ding{55}}}%

\usepackage{subcaption}

\newcommand{\R}{\mathbb{R}}
\newcommand{\eqdef}{\stackrel{\text{def}}{=}}

\def\<#1,#2>{\left\langle #1,#2\right\rangle}

\usepackage{xspace}

\newcommand{\algname}[1]{{\sf  #1}\xspace}

\usepackage[colorinlistoftodos,bordercolor=orange,backgroundcolor=orange!20,linecolor=orange,textsize=scriptsize]{todonotes}

\renewcommand{\Re}{\mathrm{Re}}

\newcommand{\Sp}{\mathrm{Sp}}
\newcommand{\Id}{\mathrm{Id}}

\newcommand{\Tr}{\mathrm{Tr}}

\newcommand{\cM}{{\cal M}}
\newcommand{\cN}{{\cal N}}
\newcommand{\cO}{{\cal O}}

\newcommand{\CC}{\mathbb{C}}

\newcommand{\tx}{\widetilde{x}}

\usepackage{hyperref}
\graphicspath{{../plots/}}

\usepackage{makecell}

\usepackage{accents}
\newlength{\dhatheight}

\usepackage{pgfplotstable} 
\usetikzlibrary{automata, positioning, arrows, shapes, fit, calc, intersections}
\usepgfplotslibrary{statistics}
\include{figure}

\usepackage{enumitem}

\begin{document}

\twocolumn[
\icmltitle{Convergence of Proximal Point and Extragradient-Based Methods\\ Beyond Monotonicity: the Case of Negative Comonotonicity}



\icmlsetsymbol{equal}{*}

\begin{icmlauthorlist}
\icmlauthor{Eduard Gorbunov}{mbzuai,mipt}
\icmlauthor{Adrien Taylor}{inria}
\icmlauthor{Samuel Horv\'ath}{mbzuai}
\icmlauthor{Gauthier Gidel}{udem,cifar}
\end{icmlauthorlist}

\icmlaffiliation{mbzuai}{Mohamed bin Zayed University of Artificial Intelligence, UAE}
\icmlaffiliation{mipt}{Moscow Institute of Physics and Technology, Russia (part of this work was done while the author was a researcher at MIPT)}
\icmlaffiliation{udem}{Universit\'e de Montr\'eal and Mila, Canada}
\icmlaffiliation{cifar}{Canada CIFAR AI Chair}
\icmlaffiliation{inria}{INRIA \& D.I.\ \'Ecole Normale Sup\'erieure, CNRS \& PSL Research University, France}

\icmlcorrespondingauthor{Eduard Gorbunov}{eduard.gorbunov@mbzuai.ac.ae}

\icmlkeywords{Machine Learning, ICML}

\vskip 0.3in
]



\printAffiliationsAndNotice{}  

\begin{abstract}
    Algorithms for min-max optimization and variational inequalities are often studied under monotonicity assumptions. Motivated by non-monotone machine learning applications, we follow the line of works \citep{diakonikolas2021efficient,lee2021fast,pethick2022escaping, bohm2022solving} aiming at going beyond monotonicity by considering the weaker \emph{negative comonotonicity} assumption. In this work, we provide tight complexity analyses for the Proximal Point (\ref{eq:PP}), Extragradient (\ref{eq:EG}), and Optimistic Gradient (\ref{eq:OG}) methods in this setup, closing several questions on their working guarantees beyond monotonicity. In particular, we derive the first non-asymptotic convergence rates for \ref{eq:PP} under negative comonotonicity and star-negative comonotonicity and show their tightness via constructing worst-case examples; we also relax the assumptions for the last-iterate convergence guarantees for \ref{eq:EG} and \ref{eq:OG} and prove the tightness of the existing best-iterate guarantees for \ref{eq:EG} and \ref{eq:OG} via constructing counter-examples.
\end{abstract}

\section{Introduction}
\looseness=-1
The study of efficient first-order methods for solving variational inequality problems~(VIP) have known a surge of interest due to the development of recent machine learning~(ML) formulations involving multiple objectives. VIP appears in various ML tasks such as robust learning~\citep{ben2009robust}, adversarial training~\citep{madry2018towards}, Generative Adversarial Networks~\citep{goodfellow2014generative}, or games with decision-dependent data~\citep{narang2022multiplayer}.
In this work, we focus on unconstrained VIPs\footnote{We refer to \citep{gidel2019variational} for the details on how these formulations appear in the real-world problems.}, which we state formally in the slightly more general form of an \emph{inclusion problem}:
\begin{equation}
    \text{find } x^* \in \R^d \text{ such that } 0\in F(x^*), \tag{IP} \label{eq:MI}
\end{equation}
where $F:\R^d \rightrightarrows \R^d$ is some (possibly set-valued) mapping. In the sequel, we use the slightly abusive shorthand notation $F(x)$ to denote any particular image of $x$ by the mapping $F$, independently of $F$ being single-valued of not.

Among the main simple first-order methods under consideration for such problems, the extragradient method~(\ref{eq:EG})~\citep{korpelevich1976extragradient} and the optimistic gradient method~(\ref{eq:OG})~\citep{popov1980modification} occupy an important place. These two algorithms have been traditionally analyzed under the assumption that the considered operator is monotone and Lipschitz~\citep{korpelevich1976extragradient,popov1980modification} and are often interpreted as an approximation to the proximal point (\ref{eq:PP}) method~\citep{nemirovski2004prox,mokhtari2019proximal}. \ref{eq:PP} can be formally stated as an implicit iterative method generating a sequence $x^1,x^2,\ldots \in\R^d$ when initiated at some $x^0\in\R^d$:
\[     x^{k+1} = x^k - \gamma F(x^{k+1}), \tag{\algname{PP}} \]
for some well-chosen stepsize $\gamma\in\R$. When $F$ is single-valued, one can instead use explicit methods such as~\ref{eq:EG}:
\begin{equation}\tag{\algname{EG}}
    \begin{aligned}
    \tx^{k} &= x^k  - \gamma_1 F(x^k),\\
    x^{k+1} &= x^k - \gamma_2 F(\tx^k),
    \end{aligned}\quad \forall k \geq 0,
\end{equation}
or~\ref{eq:OG} with the additional initialization $\tx^{0}=x^0$:
\begin{equation}\tag{\algname{OG}}
    \begin{aligned}
    \tx^{k} &= x^k  - \gamma_1 F(\tx^{k-1}),\quad \forall k > 0,\\
    x^{k+1} &= x^k - \gamma_2 F(\tx^k),\quad \forall k \geq 0,
    \end{aligned}
\end{equation}
where $\gamma_1,\gamma_2\in\R$ are some well-chosen stepsizes. For examples of the usage of extragradient-based methods in practice, we refer to \citep{daskalakis2018training} who use a variant of \ref{eq:OG} with \algname{Adam} \citep{kingma2014adam} estimators to train WGAN \citep{gulrajani2017improved} on CIFAR10 \citep{krizhevsky2009learning}, \citep{brown2019solving} where extragradient-based methods were applied in regret matching, \citep{farina2019stable} for the application to counterfactual regret minimization, and \citep{anagnostides2022last} where these methods were used for training agents to play poker.

Interestingly, until recently, the convergence rate for the \emph{last iterate} of neither \ref{eq:EG} nor \ref{eq:OG} were known even when $F$ is (maximally) monotone and Lipschitz. First results in this direction were obtained by~\citet{golowich2020last,golowich2020tight} under some additional assumptions (namely the Jacobian of $F$ being Lipschitz). Later, \citet{gorbunov2021extragradient,gorbunov2022last,cai2022tight} closed this question by proving the tight worst-case last iterate convergence rate of these methods under monotonicity and Lipschitzness of~$F$. 

As some important motivating applications involve deep neural networks, the operator $F$ under consideration is typically not monotone. However, for general non-monotone problems approximating first-order locally optimal solutions can be intractable \citep{daskalakis2021complexity, diakonikolas2021efficient}. Thus, it is natural to consider assumptions on structured non-monotonicity. Recently~\citet{diakonikolas2021efficient} proposed to analyse \ref{eq:EG} using a weaker assumption than the traditional monotonicity.  In the sequel, this assumption is referred to as \emph{$\rho$-negative comonotonicity} (with $\rho \geq 0$). That is, for all $x,y \in \R^d$, the operator $F$ satisfies:
\begin{equation}
    \langle F(x) - F(y), x- y \rangle \geq -\rho \|F(x) - F(y)\|^2. \label{eq:rho_neg_comon}
\end{equation}

A number of works have followed the idea of~\citet{diakonikolas2021efficient} and considered~\eqref{eq:rho_neg_comon} as their working assumption, see, e.g.,~\citep{yoon2021accelerated,lee2021fast,luo2022last,cai2022accelerated,gorbunov2022clipped}. Albeit being a reasonable first step toward the understanding of the behavior of algorithms for~\eqref{eq:MI} beyond $F$ being monotone, it remains unclear by what means the $\rho$-negative comonotonicity assumption is general enough to capture complex non-monotone operators. This question is crucial for developing a clean optimization theory that can fully encompass ML applications involving neural networks. 

To the best of our knowledge, \emph{$\rho$(-star)-negative comonotonicity} is the weakest known assumption under which extragradient-type methods can be analyzed for solving~\eqref{eq:MI}. The first part of this work is devoted to providing simple interpretations of this assumption. Then, we close the problem of studying the convergence rate of the \ref{eq:PP} method in this setting, the base ingredient underlying most algorithms for solving~\eqref{eq:MI} (which are traditionally interpreted as approximations to \ref{eq:PP}, see~\citep{nemirovski2004prox}). That is, we provide upper and lower convergence bounds as well as a tight conditions on its stepsize for~\ref{eq:PP} under negative comonotonicity. We eventually consider the last-iterate convergence of \ref{eq:EG} and \ref{eq:OG} and provide an almost complete picture in that case, listing the remaining open questions.

Before moving to the next sections, let us mention that many of our results were discovered using the performance estimation approach, first coined by~\cite{drori2012performance} and formalized by~\cite{taylor2017smooth,2017taylor}. The operator version of the framework is due to~\citep{ryu2020operator}. We used the framework through the packages PESTO~\citep{taylor2017performance} and PEPit~\citep{goujaud2022pepit}, thereby providing a simple way to validate our results numerically.

\subsection{Preliminaries}

In the context of~\eqref{eq:MI}, we refer to $F$ as being $\rho$-star-negative comonotone ($\rho \geq 0$) -- a relaxation\footnote{For the example of star-negative comonotone operator that is not negative comonotone we refer to \citep[Section 5.1]{daskalakis2020independent} and \citep[Section 2.2]{diakonikolas2021efficient}.} of \eqref{eq:rho_neg_comon} -- if for all $x \in \R^d$ and $x^*$ being a solution to~\eqref{eq:MI}, we have:
\begin{equation}
    \langle F(x), x- x^* \rangle \geq -\rho \|F(x)\|^2. \label{eq:rho_star_neg_comon}
\end{equation}
Furthermore, similar to monotone operators (see~\cite{bauschke2011convex} or~\cite{ryu2020large} for details), we assume that the mapping $F$ is \emph{maximal} in the sense that its graph is not strictly contained in the graph of any other $\rho$-negative comonotone operator (resp., $\rho$-star-negative comonotone), which ensures the corresponding proximal operator used in the sequel to be well-defined. Some examples of star-negative comonotone operators are given in \citep[Appendix C]{pethick2022escaping}. Moreover, if $F$ is star-monotone or quasi-strongly monotone \citep{loizou2021stochastic}, then $F$ is also star-negative comonotone. The examples of star-monotone/quasi-strongly monotone operators that are not monotone are given in \citep[Appendix A.6]{loizou2021stochastic}. Next, there are some studies of the eigenvalues of the Jacobian around the equilibrium of GAN games \citep{mescheder2018training, nagarajan2017gradient, berard2019closer}. These studies imply that the corresponding variational inequalities are locally quasi-strongly monotone. Finally, when $F$ is $L$-Lipschitz it satisfies $\langle F(x), x - x^* \rangle \geq - L\|x - x^*\|^2$. If in addition $\|F(x)\| \geq \eta \|x - x^*\|$ for some $\eta > 0$ (meaning that $F$ changes not ``too slowly''), then $\langle F(x), x - x^* \rangle \geq - \frac{L}{\eta^2}\|F(x)\|^2$, i.e., condition \eqref{eq:rho_star_neg_comon} holds with $\rho = \frac{L}{\eta^2}$.

For the analysis of the \ref{eq:EG} and \ref{eq:OG}, we further assume $F$ to be $L$-Lipschitz, meaning that for all $x,y \in \R^d$:
\begin{equation}
    \|F(x) - F(y)\| \leq L\|x - y\|. \label{eq:Lipschitzness}
\end{equation}
Note that in that case, $F$ is a single-valued mapping. In this case, \ref{eq:MI} transforms into a variational inequality:
\begin{equation}
    \text{find } x^* \in \R^d \text{ such that } F(x^*) = 0. \tag{VIP} \label{eq:VIP}
\end{equation}

\subsection{Related Work}
\begin{table*}[t]
    \centering
    \scriptsize
    \caption{\footnotesize Known and new $\cO\left(\nicefrac{1}{N}\right)$ convergence results for \ref{eq:PP}, \ref{eq:EG} and \ref{eq:OG}. Notation: NC = negative comonotonicity, SNC = star-negative comonotonicity, $L$-Lip. = $L$-Lipschitzness. Whenever the derived results are completely novel or extend the existing ones, we highlight them in green.}
    \label{tab:res_comparison}
    \vspace{-0.2cm}
    \begin{threeparttable}
        \begin{tabular}{|c|c l c c||c|}
        \hline
        Method & Setup & $\rho \in $ & Convergence & Reference & Counter-/Worst-case examples? \\
        \hline\hline
        \multirow{2}{*}{\begin{tabular}{c}
            \ref{eq:PP}\tnote{\color{blue}(1)}
        \end{tabular}} & NC &$[0, +\infty)$ & Last-iterate & Theorem~\ref{thm:PP_convergence} & \cellcolor{bgcolor2}Theorem~\ref{thm:PP_worst_case} (worst-case example) \& \ref{thm:PP_counter_example} (diverge for $\gamma \leq 2\rho$)\\
        & SNC &$[0, +\infty)$ & Best-iterate & Theorem~\ref{thm:PP_convergence} & \cellcolor{bgcolor2}Theorem~\ref{thm:PP_worst_case} (worst-case example) \& \ref{thm:PP_counter_example} (diverge for $\gamma \leq 2\rho$)\\
        \hline\hline
        \multirow{5.2}{*}{\begin{tabular}{c}
            \ref{eq:EG}
        \end{tabular}} & NC + $L$-Lip. & $[0, \nicefrac{1}{16L})$ & Last-iterate & \citep{luo2022last} & \xmark\\
        & NC + $L$-Lip. &\cellcolor{bgcolor2}$[0, \nicefrac{1}{8L})$ & Last-iterate &  \cellcolor{bgcolor2} Theorem \ref{thm:EG_convergence}& \cellcolor{bgcolor2} Theorem \ref{thm:EG_counter_example} (diverge for $\rho \geq \nicefrac{1}{2L}$ and any $\gamma_1,\gamma_2 > 0$)\\
        & SNC + $L$-Lip. & $[0, \nicefrac{1}{8L})$ & Best-iterate & \citep{diakonikolas2021efficient} & \xmark\\
        & SNC + $L$-Lip. & $[0, \nicefrac{1}{2L})$ & Best-iterate & \citep{pethick2022escaping} & Theorem 3.4 (diverge for $\gamma_1 = \nicefrac{1}{L}$ and $\rho \geq \nicefrac{(1-L\gamma_2)}{2L}$)\\
        & SNC + $L$-Lip. & $[0, \nicefrac{1}{2L})$ & Best-iterate & Theorem \ref{thm:EG_convergence} \tnote{\color{blue}(2)} & \cellcolor{bgcolor2} Theorem \ref{thm:EG_counter_example} (diverge for $\rho \geq \nicefrac{1}{2L}$ and any $\gamma_1,\gamma_2 > 0$)\\
        \hline\hline
        \multirow{4.3}{*}{\begin{tabular}{c}
            \ref{eq:OG}
        \end{tabular}} & NC + $L$-Lip. & $[0, \nicefrac{8}{(27\sqrt{6}L)})$ & Last-iterate & \citep{luo2022last} & \xmark\\
        & NC + $L$-Lip. &\cellcolor{bgcolor2}$[0, \nicefrac{5}{62L})$ & Last-iterate &\cellcolor{bgcolor2} Theorem \ref{thm:OG_convergence}& \cellcolor{bgcolor2} Theorem \ref{thm:OG_counter_example} (diverge for $\rho \geq \nicefrac{1}{2L}$ and any $\gamma_1,\gamma_2 > 0$)\\
        & SNC + $L$-Lip. & $[0, \nicefrac{1}{2L})$ & Best-iterate & \citep{bohm2022solving} & \xmark\\
        & SNC + $L$-Lip. & $[0, \nicefrac{1}{2L})$ & Best-iterate & Theorem \ref{thm:OG_convergence} \tnote{\color{blue}(2)}& \cellcolor{bgcolor2} Theorem \ref{thm:OG_counter_example} (diverge for $\rho \geq \nicefrac{1}{2L}$ and any $\gamma_1,\gamma_2 > 0$)\\
        \hline
    \end{tabular}
    \begin{tablenotes}
        {\footnotesize\item [{\color{blue}(1)}] The best-iterate convergence result can be obtained \citep[Lemma~2]{iusem2003inexact}, and the last-iterate convergence result can also be derived from the non-expansiveness of \ref{eq:PP} update \citep[Proposition 3.13 (iii)]{bauschke2021generalized}. At the moment of writing our paper, we were not aware of these results.
        \item [{\color{blue}(2)}] Although these results are not new for the best-iterate convergence of \ref{eq:EG} and \ref{eq:OG}, the proof techniques differ from prior works.
        }
    \end{tablenotes}
    \end{threeparttable}
\end{table*}
\paragraph{Last-iterate convergence rates in the monotone case.} Several recent theoretical advances focus on the last-iterate convergence of the methods for solving \ref{eq:MI}/\ref{eq:VIP} with \emph{monotone} operator $F$. In particular, \citet{he2015convergence} derive the last-iterate $\cO(\nicefrac{1}{N})$ rate\footnote{Here and below we mean the rates of convergence in terms of the squared residual $\|x^N - x^{N-1}\|^2$ in the case of set-valued operators and $\|F(x^N)\|^2$ in the case of single-valued ones.} for \ref{eq:PP} and \citet{gu2020tight} show its tightness. Under the additional assumption of Lipschitzness of $F$ and of its Jacobian, \citet{golowich2020last,golowich2020tight} obtain last-iterate $\cO(\nicefrac{1}{N})$ convergence for \ref{eq:EG} and \ref{eq:OG} and prove matching lower bounds for them. Next, \citet{gorbunov2021extragradient,gorbunov2022last,cai2022tight} prove similar upper bounds for \ref{eq:EG}/\ref{eq:OG} without relying on the Lipschitzness (and even existence) of the Jacobian of $F$. Finally, for this class of problems one can design (accelerated) methods with provable $\cO(\nicefrac{1}{N^2})$ last-iterate convergence rate \citep{yoon2021accelerated, bot2022fast, tran2021halpern, tran2022connection}. Although $\cO(\nicefrac{1}{N^2})$ is much better than $\cO(\nicefrac{1}{N})$, \ref{eq:EG}/\ref{eq:OG} are still more popular due to their higher flexibility. Moreover, when applied to non-monotone problems the mentioned accelerated methods may be attracted to ``bad'' stationary points, see, e.g.,~\citep[Example 1.1]{gorbunov2022last}.

\paragraph{Best-iterate convergence under $\rho$-star-negative comonotonicity.} The convergence of \ref{eq:EG} is also studied under $\rho$-star-negative comonotonicity (and $L$-Lipschitzness): \citet{diakonikolas2021efficient} prove best-iterate $\cO(\nicefrac{1}{N})$ convergence of \ref{eq:EG} with $\gamma_2 < \gamma_1$ for any $\rho < \nicefrac{1}{8L}$ and \citet{pethick2022escaping} derive a similar result for any $\rho < \nicefrac{1}{2L}$. Moreover, \citet{pethick2022escaping} show that \ref{eq:EG} is not necessary convergent when $\gamma_1 = \nicefrac{1}{L}$ and $\rho \geq \nicefrac{(1-L\gamma_2)}{2L}$. \citet{bohm2022solving} prove best-iterate $\cO(\nicefrac{1}{N})$ convergence of \ref{eq:OG} for $\rho < \nicefrac{1}{2L}$, i.e., for the same range of $\rho$ as in the best-known result for \ref{eq:EG}.

\paragraph{Last-iterate convergence under $\rho$-negative comonotonicity.} In a very recent work, \citet{luo2022last} prove the first last-iterate $\cO(\nicefrac{1}{N})$ convergence results for \ref{eq:EG} and \ref{eq:OG} applied to solve \ref{eq:VIP} with $\rho$-negative comonotone $L$-Lipschitz operator. Both results rely on the usage of $\gamma_1 = \gamma_2$. Next, for \ref{eq:EG} the result from \citep{luo2022last} requires $\rho < \nicefrac{1}{16L}$ and for \ref{eq:OG} the corresponding result is proven for $\rho < \nicefrac{4}{(27\sqrt{6}L)}$. In contrast, for the accelerated (anchored) version of \ref{eq:EG} \citet{lee2021fast} prove $\cO(\nicefrac{1}{N^2})$ last-iterate convergence rate for any $\rho < \nicefrac{1}{2L}$, which is a larger range of $\rho$ than in the known results for \ref{eq:EG}/\ref{eq:OG} from \citep{luo2022last}.

\subsection{Contributions}\label{s:contrib}

\textbf{$\diamond$ Spectral viewpoint on negative comonotonicity.} Our work provides a spectral interpretation of negative comonotonicity, shedding some light on the relation between this assumption and classical monotonicity, Lipschitzness, and cocoercivity. 

\textbf{$\diamond$ Closer look at the convergence of Proximal Point method.} We derive $\cO(\nicefrac{1}{N})$ last-iterate and best-iterate convergence rates for \ref{eq:PP} under negative comonotonicity and star-negative comonotonicity assumptions, respectively. These results follow from existing ones \citep{iusem2003inexact, bauschke2021generalized}. However, we go further and show the tightness of the derived results via constructing matching worst-case examples and also propose counter-examples for the case when the stepsize is smaller than $2\rho$.


\looseness=-1
\textbf{$\diamond$ New results for Extragradient-Based Methods.} We derive $\cO(\nicefrac{1}{N})$ last-iterate convergence of \ref{eq:EG} and \ref{eq:OG} under milder assumptions on the negative comonotonicity parameter $\rho$ than in the prior work by \citet{luo2022last}, see the details in Table~\ref{tab:res_comparison}. We also provide alternative analyses of the best-iterate convergence of \ref{eq:EG} and \ref{eq:OG} under star-negative comonotonicity and recover the best-known results in this case \citep{pethick2022escaping, bohm2022solving}. Finally, we show that the range of allowed $\rho$ cannot be improved for \ref{eq:EG} and \ref{eq:OG} via constructing counter-examples for these methods.

\looseness=-1
\textbf{$\diamond$ Constructive proofs.} We derive the proofs for the last-iterate convergence of \ref{eq:PP}, \ref{eq:EG}, and \ref{eq:OG} as well as worst-case examples for \ref{eq:PP} using using the performance estimation technique \citep{drori2012performance,taylor2017smooth,2017taylor}. In particular, it required us to extend some theoretical and program tools to handle negative comonotone and star-negative comonotone problems; see the details in App.~\ref{appendix:PEP} and Github-repository \url{https://github.com/eduardgorbunov/Proximal_Point_and_Extragradient_based_methods_negative_comonotonicity}, containing the codes for generating worst-case examples for \ref{eq:PP}, numerical verification of the derived results and symbolical verification of certain technical derivations. We believe that these tools are important on its own and can be applied in future works studying the convergence of different methods under negative comonotonicity.

\section{A Closer Look at Negative Comonotonicity}\label{sec:neg_comon}

Negative comonotonicity (also known as cohypomonotonicity) was originally introduced as a relaxation of monotonicity that is sufficient for the convergence of \ref{eq:PP} \citep{pennanen2002local}. This assumption is relatively weak: one can show that $F$ is $\rho$-negative comonotone in a neighborhood of solution $x^*$ for large enough $\rho$, if the (possibly set-valued) operator $F^{-1}:\R^d \rightrightarrows \R^d$ has a Lipschitz localization around $(0,x^*) \in G_{F^{-1}}$, where $G_{F^{-1}}$ denotes the graph of $F^{-1}$ \citep[Proposition 7]{pennanen2002local}. The next lemma characterizes negative comonotone operators; it is technically very close to~\citep[Proposition 4.2]{bauschke2011convex} (on cocoercive operators).
\begin{lemma}\label{lem:expansiveness_of_neg_comon_operator}
    $F:\R^d \rightrightarrows \R^d$ is maximally $\rho$-negative comonotone ($\rho\geq 0$) if and only if operator $\Id + 2\rho F$ is expansive.
\end{lemma}
The proof of this lemma follows directly from the definition of negative comonotonicity. Among others, it implies the following result about the spectral properties of the Jacobian of negative comonotone operator (when it exists).
\begin{theorem}\label{thm:spectral_viewpoint_on_neg_comon}
    Let $F:\R^d \to \R^d$ be a continuously differentiable. Then, the following statements are equivalent:
    \begin{itemize}
        \item $F$ is $\rho$-negative comonotone,
        \item $\Re(\nicefrac{1}{\lambda}) \geq -\rho$ for all $\lambda \in \Sp(\nabla F(x))$, $\forall x \in \R^d$.
    \end{itemize}
\end{theorem}

We notice that the above theorem holds for any continuously differentiable operator $F$. In the case of the linear operator $F$, this result is known \citep[Proposition 5.1]{bauschke2021generalized}. The condition $\Re(\nicefrac{1}{\lambda}) \geq -\rho$ means that $\lambda$ lies \emph{outside} the disc in $\CC$ centered at $-\nicefrac{1}{2\rho}$ and having radius $\nicefrac{1}{2\rho}$, see Figure~\ref{fig:spectral_neg_comon}. In particular, for the case of twice differentiable functions $\rho$-negative comonotonicity forbids the Hessian to have eigenvalues in $(-\nicefrac{1}{\rho}, 0)$, i.e., eigenvalues of the Hessian have to be either negative with sufficiently large absolute value or non-negative. An alternate interpretation of Figure~\ref{fig:spectral_neg_comon} can be formally made in terms of scaled relative graphs, see~\citep{ryu2022scaled}; see also older references using such illustrations~\citep{eckstein1989splitting,eckstein1992douglas}, or~\citep[arXiv version $1$ to $3$]{giselsson2016linear}.

\begin{figure}[t]
	\centering
	\newcounter{t}
	\newcounter{t2}
	\newcounter{L}
	\newcounter{mu}
	\setcounter{t}{90}
	\setcounter{L}{18}
	\setcounter{mu}{5}
	\setcounter{t2}{60}
	\begin{tikzpicture}[scale=1]
	\begin{axis}[
	grid=major,
	axis equal image,
    ytick = {-\value{L}/2,0,\value{L}/2},
    yticklabels={-$\tfrac{i}{2\rho}$,$0i$,$\tfrac{i}{2\rho}$},
    xtick = {-\value{L},-\value{L}/2,0},
	xticklabels={-$\tfrac{1}{\rho}$,-$\tfrac{1}{2\rho}$,$0$},
	xmin={-\value{L} - 3},   xmax=6,
	ymin={-\value{L}/2 - 2},   ymax={\value{L}/2 + 2},
	]
	\draw [domain=-180:180, draw=white, fill=red!50!white, thick, fill opacity=0.5, samples=65] plot (axis cs: -{\value{L}*(1 + cos(\value{t}))/2 + cos(\x) * \value{L}*(1 - cos(\value{t}))/2}, {sin(\x) * \value{L}*(1 - cos(\value{t}))/2});
	\node[anchor=north west] (mu) at (axis cs: {\value{L} * cos(\value{t})}, 0) {$0$};
	\draw [fill=black] (axis cs: {\value{L} * cos(\value{t})}, 0) circle (1pt);
	\draw [fill=black] (axis cs: {-\value{L}}, 0) circle (1pt);
	\node [circle, anchor=north east] (L) at (axis cs: -{\value{L}}, 0) {};
	\draw [fill=black] (axis cs: {-\value{L}/2}, 0) circle (1pt);
	\node [circle, anchor=north] (L) at (axis cs: -{\value{L}/2}, 0) {};
	\node [circle, anchor=south] at (axis cs: {-\value{L}/2}, -9 ) {\small No $\rho$-negative comonotonicity};
	\end{axis}
	\end{tikzpicture}
	\caption{Visualization of Theorem~\ref{thm:spectral_viewpoint_on_neg_comon}. Red open disc corresponds to the constraint $\Re(\nicefrac{1}{\lambda}) < -\rho$ that defines the set such that all eigenvalues the Jacobian of $\rho$-negative comonotone operator should lie outside this set.
	}\label{fig:spectral_neg_comon}
\end{figure}
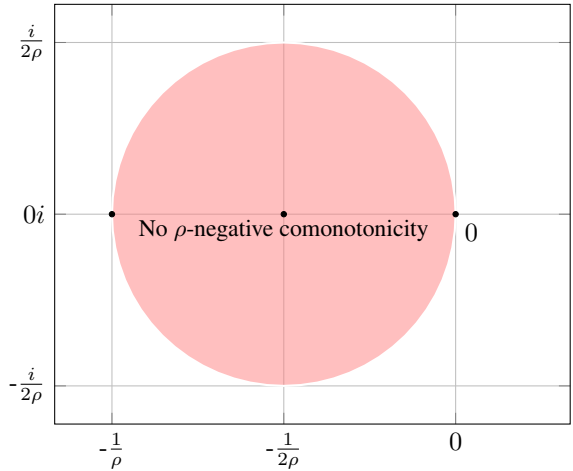

Finally, we touch the following informal question: \emph{to what extent negative comonotone operators are non-monotone?} To formalize a bit we consider a way more simpler question: \emph{can negative comonotone operator have isolated zeros/solutions of \ref{eq:VIP}?} Unfortunately, the answer is no.

\begin{theorem}[Corollary 3.15 from \citep{bauschke2021generalized}\footnote{We were not aware of the results from \citep{bauschke2021generalized} during the work on our paper.}]\label{thm:no_isolated_minima}
    If $F:\R^d \rightrightarrows \R^d$ is maximally $\rho$-negative comonotone, then the solution set $X^* = F^{-1}(0)$ is convex.
\end{theorem}
\begin{proof}
    The proof follows from the observations provided by \citet{pennanen2002local}. First, notice that $F$ and its Yosida regularization $(F^{-1} + \rho\cdot \Id)^{-1}$ have the same set of the solutions: $((F^{-1} + \rho\cdot \Id)^{-1})^{-1}(0) = (F^{-1} + \rho\cdot \Id)(0) = F^{-1}(0)$. Next, by definition \eqref{eq:rho_neg_comon} we have that maximal $\rho$-negative comonotonicity of $F$ implies maximal monotonicity of $F^{-1} + \rho\cdot \Id$ that is equivalent to maximal monotonicity of $(F^{-1} + \rho\cdot \Id)^{-1}$. Since the set of zeros of maximal monotone operator is convex \citep[Proposition 23.39]{bauschke2011convex}, we have the result.
\end{proof}

Therefore, despite its apparent generality, negative comonotonicity is not satisfied (globally) for the many practical tasks that have isolated optima. Nevertheless, studying the convergence of traditional methods under negative comonotonicity can be seen as a natural step towards understanding their behaviors in more complicated non-monotonic cases.

\section{Proximal Point Method}\label{sec:prox_point}

\begin{figure}[ht!]
\centering
\begin{subfigure}[b]{0.4\textwidth}
    \centering
    \includegraphics[width=\textwidth]{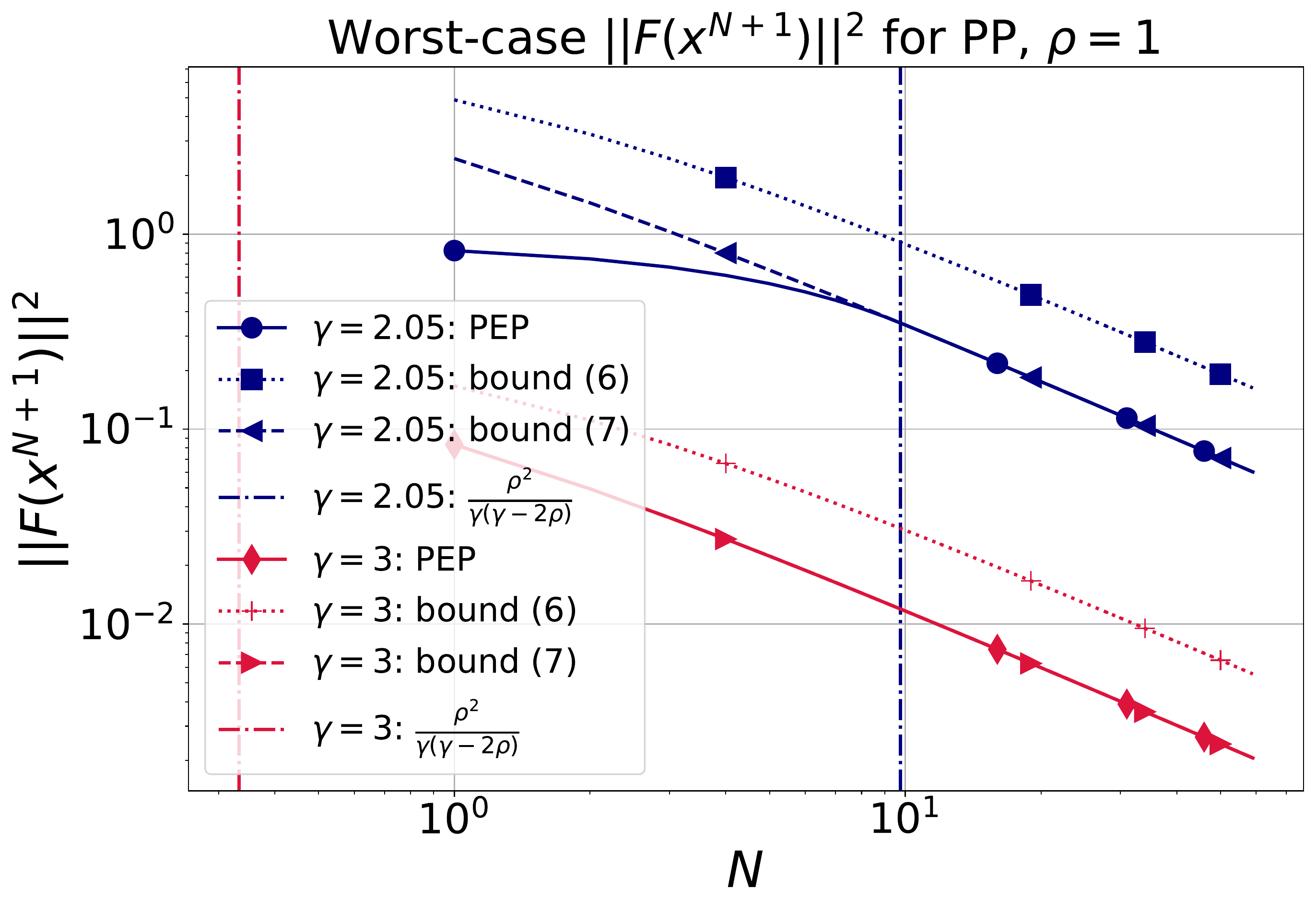}
    \caption{Worst-case $\|F(x^{N+1})\|^2$}\label{fig:1a}
\end{subfigure}\vspace{.1cm}
\begin{subfigure}[b]{0.4\textwidth}
    \centering
    \includegraphics[width=\textwidth]{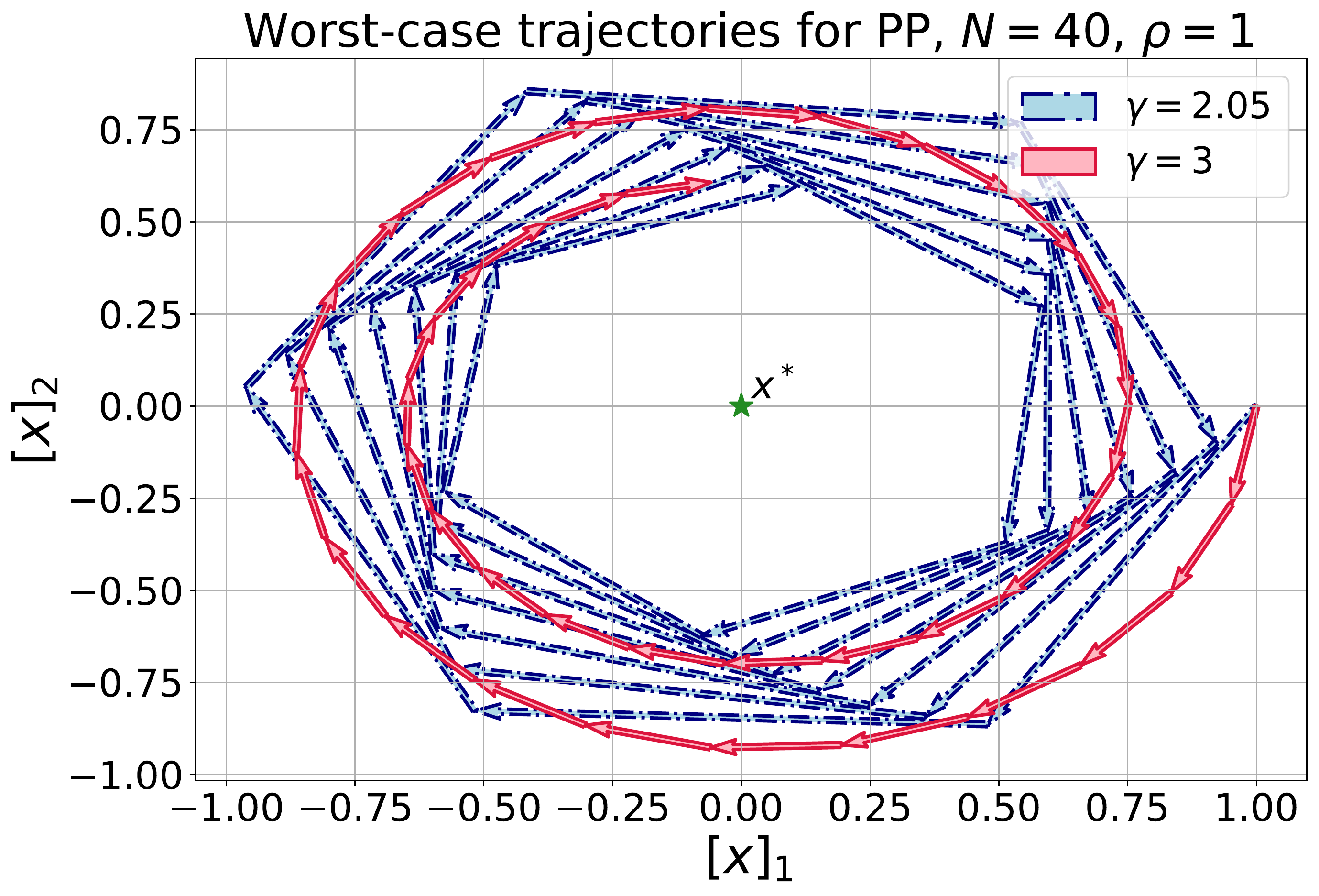}
    \caption{Worst-case trajectories}\label{fig:1b}
\end{subfigure}\vspace{.1cm}
\begin{subfigure}[b]{0.4\textwidth}
    \centering
    \includegraphics[width=\textwidth]{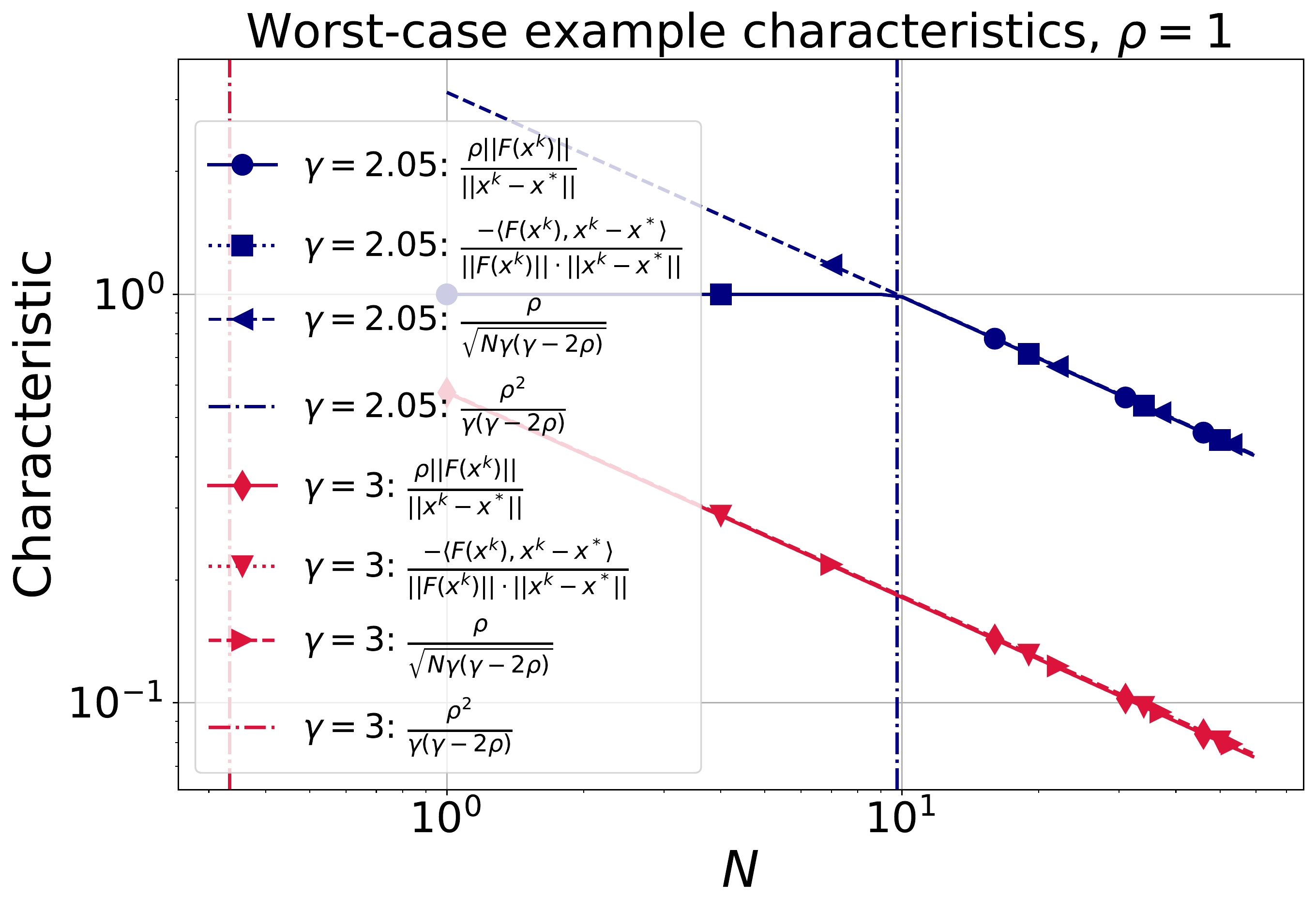}
    \caption{Study of the norms and angles}\label{fig:1c}
\end{subfigure}
\caption{\small In~(a), we report the solution of \eqref{eq:PEP_for_PP} for different values of $\gamma$ and $N$. The plot illustrates that for the considered range of $N$ and values of $\gamma$ \ref{eq:PP} converges as $\cO(\nicefrac{1}{N})$ in terms of $\|F(x^{N})\|^2$. In~(b), we show the worst-case trajectories of \ref{eq:PP} for $N = 40$. The form of trajectories hints that the worst-case operator is a rotation operator. For each  particular choice of $N$ and $\gamma > 2\rho$ we observed numerically that quantities $\nicefrac{\rho\|F(x^k)\|^2}{\|x^k - x^*\|}$ and $\nicefrac{-\langle F(x^k) ,  x^k - x^* \rangle}{(\|F(x^k)\|\cdot \|x^k - x^*\|)}$ remain the same during the run of the method (the standard deviation of arrays $\{\nicefrac{\rho\|F(x^k)\|^2}{\|x^k - x^*\|}\}_{k=1}^N$ and $\{\nicefrac{-\langle F(x^k) ,  x^k - x^* \rangle}{(\|F(x^k)\|\cdot \|x^k - x^*\|)}\}_{k=1}^N$ is of the order $10^{-6}-10^{-7}$). Finally, in~(c), we illustrate that these characteristics coincide with $\nicefrac{\rho}{\sqrt{N\gamma(\gamma - 2\rho)}}$ as long as the total number of steps $N$ is sufficiently large ($N \geq \max\{\nicefrac{\rho^2}{\gamma(\gamma-2\rho)},1\}$).}
\label{fig:1}
\end{figure}

In this section, we consider Proximal Point method \citep{martinet1970regularisation, rockafellar1976monotone}, which is usually written as $x^{k+1}=(F+\gamma\, \Id)^{-1}\,x^k$ (where we assume here that $\gamma>0$ is large enough so that the iteration is well and uniquely defined) or equivalently:
\begin{equation}
    x^{k+1} = x^k - \gamma F(x^{k+1}). \tag{\algname{PP}} \label{eq:PP}
\end{equation}
In particular, for given values of $N\in\mathbb{N}$, $R > 0$, $\rho > 0$, and $\gamma > 0$ we focus on the following question: \emph{what guarantees can we prove on $\|x^N - x^{N-1}\|^2$ (in particular: as a function of $N$), where $\{x^k\}_{k=0}^{N}$ is generated by \ref{eq:PP} with stepsize $\gamma$ after $N\geq 1$ iterations of solving \ref{eq:MI} with $F: \R^d \rightrightarrows \R^d$ being $\rho$-negative comonotone and $\|x^0 - x^*\|^2 \leq R^2$?} This kind of question can naturally be reformulated as an explicit optimization problem looking for worst-case problem instances, often referred to as \emph{performance estimation problems} (PEPs), as introduced and formalized in~\citep{drori2012performance,taylor2017smooth,2017taylor}:
\begin{eqnarray}
    \max\limits_{F, x^0}&&\|x^N - x^{N-1}\|^2 \label{eq:PEP_for_PP}\\
    \text{s.t.}&& F \text{ satisfies } \eqref{eq:rho_neg_comon},\notag\\
    &&\|x^0 - x^*\|^2 \leq R^2,\; 0 \in  F(x^*),\notag\\
    && x^{k+1} = x^k - \gamma F(x^{k+1}),\quad  k=0,1,\ldots,N-1.\notag
\end{eqnarray}
As we show in Appendix~\ref{appendix:PEP},~\eqref{eq:PEP_for_PP} can be \emph{formulated} as a semidefinite program (SDP). For constructing and solving this SDP problem corresponding to \eqref{eq:PEP_for_PP} numerically, one can use the PEPit package \citep{goujaud2022pepit} (after adding the class of $\rho$-negative comonotone operators to it), which thereby allows constructing worst-case guarantees and examples, numerically. Figure~\ref{fig:1a} shows the numerical results obtained by solving~\eqref{eq:PEP_for_PP} for different values of $N$. We observe that worst-case value of \eqref{eq:PEP_for_PP} behaves as $\cO(\nicefrac{1}{N})$ similarly to the monotone case.

Motivated by these numerical results, we derive the following convergence rates for \ref{eq:PP}.
\begin{theorem}\label{thm:PP_convergence}
    Let $F: \R^d \rightrightarrows \R^d$ be maximally $\rho$-star-negative comonotone. Then, for any $\gamma > 2\rho$ the iterates produced by \ref{eq:PP} are well-defined and satisfy $\forall N \geq 1$:
    \begin{equation}
        \frac{1}{N}\sum\limits_{k=1}^N \|x^k - x^{k-1}\|^2 \leq \frac{\gamma\|x^0 - x^*\|^2}{(\gamma - 2\rho)N}. \label{eq:PP_convergence_star_neg_comon}
    \end{equation}
    If $F: \R^d \rightrightarrows \R^d$ is maximally $\rho$-negative comonotone, then for any $\gamma > 2\rho$ and any $k \geq 1$ the iterates produced by \ref{eq:PP} satisfy
    \begin{equation*}
        \|x^{k+1} - x^k\| \leq \|x^k - x^{k-1}\|
    \end{equation*}
    and for any $N \geq 1$:
    \begin{equation}
        \|x^N - x^{N-1}\|^2 \leq \frac{\gamma\|x^0 - x^*\|^2}{(\gamma - 2\rho)N}.\label{eq:PP_convergence_neg_comon}
    \end{equation}
\end{theorem}
\begin{proof}
    We start with $\rho$-star-negative comonotone case. From the update rule of \ref{eq:PP} we have
    \begin{align*}
        \|x^{k+1} - x^*\|^2 &= \|x^{k} - x^* - (x^k - x^{k+1}) \|^2\\
        &= \|x^k - x^*\|^2 - 2\langle x^k - x^*,  x^k - x^{k+1}\rangle\\
        &\quad + \|x^k - x^{k+1}\|^2\\
        &= \|x^k - x^*\|^2 - 2\langle x^{k+1} - x^*,  x^k - x^{k+1}\rangle \\
        &\quad - \|x^k - x^{k+1}\|^2.
    \end{align*}
    Since $x^k - x^{k+1} = \gamma F(x^{k+1})$, where $F(x^{k+1})$ is some value of operator $F$ at point $x^{k+1}$, we can apply $\rho$-star-negative comonotonicity and get
    \begin{align*}
        \|x^{k+1} - x^*\|^2 &\leq \|x^k - x^*\|^2 - \left(1 - \frac{2\rho}{\gamma}\right) \|x^k - x^{k+1}\|^2.
    \end{align*}
    Telescoping the above inequality for $k=0,\ldots,N-1$ and changing the index in the summation, we obtain \eqref{eq:PP_convergence_star_neg_comon}. Next, to get the last-iterate convergence we use $\rho$-negative comonotonicity \eqref{eq:rho_neg_comon} inequality written for~$x^k$ and~$x^{k+1}$:
    \begin{align*}
        \frac{1}{\gamma}\langle x^{k-1} - x^k - &(x^k - x^{k+1}), x^k - x^{k+1} \rangle \\
        &\geq -\frac{\rho}{\gamma^2}\|x^{k-1} - x^k - (x^k - x^{k+1})\|^2,
    \end{align*}
    where we use that $\nicefrac{(x^{k-1} - x^k)}{\gamma}$ and $\nicefrac{(x^{k} - x^{k+1})}{\gamma}$ belongs to the values of $F$ at points $x^k$ and $x^{k+1}$ respectively. Multiplying both sides by $\gamma^2$ and rearranging the terms, we get
    \begin{align*}
        \gamma \|x^k - x^{k+1}\|^2 &\leq \gamma\langle x^{k-1} - x^k, x^k - x^{k+1} \rangle \\
        &\quad + \rho \|x^{k-1} + x^{k+1} - 2x^k\|^2.
    \end{align*}
    Finally, using $2\langle a,b \rangle = \|a\|^2 + \|b\|^2 - \|a - b\|^2$, which holds for any $a,b \in \R^d$, and rearranging the terms, we derive
    \begin{align*}
        \frac{\gamma}{2}\|x^k - x^{k+1}\|^2 &\leq \frac{\gamma}{2}\|x^{k-1} - x^k\|^2 \\
        &\quad - \left(\frac{\gamma}{2} - \rho\right)\|x^{k-1} + x^{k+1} - 2x^k\|^2.
    \end{align*}
    Taking into account $\gamma > 2\rho$, we obtain $\|x^{k+1} - x^k\| \leq \|x^k - x^{k-1}\|$. Together with \eqref{eq:PP_convergence_star_neg_comon} it implies \eqref{eq:PP_convergence_neg_comon}.
\end{proof}

First, the result from \eqref{eq:PP_convergence_star_neg_comon} implies only best-iterate $\cO(\nicefrac{1}{N})$ rate -- this result follows from \citep[Lemma~2]{iusem2003inexact}\footnote{We were not aware of the results from \citep{iusem2003inexact} during the work on our paper.}. Such kind of guarantees are weaker the last-iterate ones but they do hold under the more general star-negative comonotonicity assumption. We notice that the result from \eqref{eq:PP_convergence_neg_comon} can be also obtained from non-expansiveness of \ref{eq:PP} update \citep[Proposition 3.13 (iii)]{bauschke2021generalized}\footnote{We were not aware of the results from \citep{bauschke2021generalized} during the work on our paper.}. Note that the guarantee~\eqref{eq:PP_convergence_neg_comon} matches the best-known guarantee for the monotone case (up to the factor $\nicefrac{\gamma}{(\gamma - 2\rho)}$) from~\citep{he2015convergence, gu2020tight}, and it is therefore natural to ask whether {it is possible to improve factor $\nicefrac{\gamma}{(\gamma - 2\rho)}$ in the convergence guarantee of \ref{eq:PP} for the $\rho$-negative comonotone case.}

To answer this question, one can use performance estimation again. In particular, using the trace heuristic for trying to find low-dimensional worst-case examples to~\eqref{eq:PEP_for_PP}, we obtain $2$-dimensional worst-case examples for different values of~$\gamma$ and~$N$, see Figure~\ref{fig:1b} and Figure~\ref{fig:1c}. These figures illustrate that the worst-case examples found numerically correspond to the scaled rotation operators (similar to~\citet{gu2020tight} but with different angles). Moreover, the rotation angle and scaling parameter have non-trivial dependencies on number of iterations. These observations lead to the following result, which shows that the multiplicative cannot be removed asymptotically as $N$ grows.

\begin{theorem}\label{thm:PP_worst_case}
    For any $\rho > 0, \gamma > 2\rho$, and $N \geq \max\{\nicefrac{\rho^2}{\gamma(\gamma-2\rho)},1\}$ there exists $\rho$-negatively comonotone single-valued operator $F: \R^d \to \R^d$ such that after $N$ iterations \ref{eq:PP} with stepsize $\gamma$ produces $x^{N+1}$ satisfying
    \begin{equation}
        \|F(x^{N+1})\|^2 \geq \frac{\|x^0 - x^*\|^2}{\gamma(\gamma - 2\rho)N\left(1 + \frac{1}{N}\right)^{N+1}}.\label{eq:PP_worst_case}
    \end{equation}
    Indeed, one can pick the two-dimensional $F:\R^2\to \R$: $F(x) = \alpha A x$ with $$A = \begin{pmatrix} \cos\theta & -\sin\theta \\ \sin\theta & \cos\theta \end{pmatrix},\quad \alpha = \frac{|\cos \theta|}{\rho}$$ for $\theta \in (\nicefrac{\pi}{2}, \pi)$ such that $\cos \theta = - \frac{\rho}{\sqrt{N\gamma(\gamma - 2\rho)}}$.
\end{theorem}
\begin{proof}
    Consider the linear operator $F(x) = \alpha A x$ described above. First, we verify its $\rho$-negative comonotonicity: for any $x, y \in \R^d$
    \begin{align*}
        \langle F(x) - F(y), x - y \rangle &= \alpha \langle A(x-y), x-y \rangle\\
        &= \alpha \|A(x-y)\|\cdot \|x-y\|\cdot \cos\theta\\
        &= - \frac{\cos^2 \theta}{\rho} \|A(x-y)\|^2\\
        &= -\rho \|F(x) - F(y)\|^2,
    \end{align*}
    where we use $\|A(x-y)\| = \|x-y\|$, since $A$ is the rotation matrix. Next, one can check that $(I + \gamma\alpha A)^{-1}$ equals
    \begin{equation*}
        \frac{1}{1\! +\! \gamma\alpha^2(\gamma\! -\! 2\rho)}\! \begin{pmatrix} \!1 \!+\! \gamma\alpha \cos\theta \!& \gamma\alpha\sin\theta \\  \!-\gamma\alpha\sin\theta & 1\! +\! \gamma\alpha \cos\theta\! \end{pmatrix}.
    \end{equation*}
    Since $x^{k+1} = (I + \gamma\alpha A)^{-1}x^k$, one can verify via direct computations that
    \begin{align*}
         \|x^{k+1}\|^2 = \frac{1}{1 + \gamma\alpha^2(\gamma - 2\rho)}\|x^k\|^2.
    \end{align*}
    Unrolling this identity for $k = N, N-1, \ldots, 0$ and using $x^* = 0$, $\|F(x^{k+1})\| = \alpha\|Ax^{k+1}\| = \alpha\|x^{k+1}\|$, we get
    \begin{equation*}
        \|F(x^{k+1})\|^2 = \alpha^2 \left(\frac{1}{1 + \gamma\alpha^2(\gamma - 2\rho)}\right)^{N+1} \|x^0 - x^*\|^2.
    \end{equation*}
    Maximizing the right-hand side in $\alpha$ we get that the optimal value is $\alpha = \nicefrac{1}{\sqrt{N\gamma(\gamma-2\rho)}}$. Since $\alpha\rho = |\cos\theta|$, we should assume that $N \geq \nicefrac{\rho^2}{\gamma(\gamma - 2\rho)}$. Plugging $\alpha = \nicefrac{1}{\sqrt{N\gamma(\gamma-2\rho)}}$ in the above formula for $\|F(x^{k+1})\|^2$ we get the result.
\end{proof}

Since $\exp(1) \leq (1+ \nicefrac{1}{N})^{N+1} \leq 4$, the above result implies the tightness (up to a multiplicative constant) of Theorem~\ref{thm:PP_convergence}. One should note again that both Theorem~\ref{thm:PP_convergence} and \ref{thm:PP_worst_case} rely on the assumption that $\gamma > 2\rho$ for the proximal operation to be well-defined. That is, these results are valid only for \emph{large enough stepsizes}. This is a relatively rare phenomenon in optimization and variational inequalities. As the next theorem states, \emph{\ref{eq:PP} is not guaranteed to converge if the stepsize is too small}, even if the proximal operation is well-defined.

\begin{theorem}\label{thm:PP_counter_example}
    For any $\rho > 0$ there exists $\rho$-negatively comonotone single-valued operator $F: \R^d \to \R^d$ such that \ref{eq:PP} does not converge to the solution of \ref{eq:VIP} for any $0 < \gamma \leq 2\rho$. In particular, one can take $F(x) = - \nicefrac{x}{\rho}$.
\end{theorem}
\begin{proof}
    First, $F(x) = - \nicefrac{x}{\rho}$ is $\rho$-negative comonotone: for any $x,y \in \R^d$ we have $\langle F(x) - F(y), x- y \rangle = -(\nicefrac{1}{\rho})\|x-y\|^2 = -\rho\|F(x) - F(y)\|^2$. 
    Next, the iterates of \ref{eq:PP} satisfy $x^{k+1} = x^k + \gamma\nicefrac{x^{k+1}}{\rho}$. If $\gamma = \rho$, the next iterate is undefined. If $\gamma = 2\rho$, then $x^{k+1} = x^k$. Finally, when $\gamma \in (0,\rho)\cup (\rho, 2\rho)$ we have $x^{k+1} = \frac{x^k}{(1 - \nicefrac{\gamma}{\rho})}$ implying $\|x^{k+1}\| > \|x^k\|$, i.e., \ref{eq:PP} diverges.
\end{proof}

As a summary, Theorem~\ref{thm:PP_convergence} and Theorem~\ref{thm:PP_counter_example} provide a complete picture of the convergence of \ref{eq:PP} under negative comonotonicity, including the upper bounds, and worst-case examples and counter-examples justifying the need of using large enough stepsizes for \algname{PP} applied to $\rho$-negative comonotone \ref{eq:MI}/\ref{eq:VIP}.

\section{Extragradient-Based Methods}\label{sec:eg_peg}
\paragraph{Extragradient.} The update rule of Extragradient method \citep{korpelevich1976extragradient} is defined as follows:
\begin{equation}\tag{\algname{EG}}\label{eq:EG}
    \begin{aligned}
    \tx^{k} &= x^k  - \gamma_1 F(x^k),\\
    x^{k+1} &= x^k - \gamma_2 F(\tx^k),
    \end{aligned}\quad \forall k \geq 0.
\end{equation}
In its pure form, \ref{eq:EG} has the same extrapolation ($\gamma_1$) and update ($\gamma_2$) stepsizes, i.e., $\gamma_1 = \gamma_2$. However, the existing analysis of \ref{eq:EG} under $\rho$-(star-)negative comonotonicity relies on the usage of $\gamma_2 < \gamma_1$ \citep{diakonikolas2021efficient, pethick2022escaping}. The following lemma sheds some light on this phenomenon.

\begin{lemma}\label{lem:EG_main_lemma}
    Let $F$ be $L$-Lipschitz and $\rho$-star-negative comonotone. Then, for any $k\geq 0$ the iterates produced by \ref{eq:EG} after $k\geq 0$ iterations satisfy
    \begin{eqnarray}
        \|x^{k+1} - x^*\|^2 &\leq& \|x^k - x^*\|^2\notag\\
        && - \gamma_2\left(\gamma_1 - 2\rho - \gamma_2\right)\|F(\tx^k)\|^2 \label{eq:EG_key_inequality_main_1}\\
        && - \gamma_1\gamma_2(1 - L^2\gamma_1^2)\|F(x^k)\|^2. \label{eq:EG_key_inequality_main_2}
    \end{eqnarray}
\end{lemma}
\begin{proof}[Proof sketch]
    The proof follows a quite standard pattern: we start with expanding the square $\|x^{k+1} - x^*\|^2$ and then rearrange the terms to get $\|x^k - x^*\|^2 - 2\gamma_2 \langle \tx^k - x^*, F(\tx^k) \rangle - 2\gamma_1\gamma_2 \langle F(x^k), F(\tx^k) \rangle + \gamma_2^2 \|F(\tx^k)\|^2$ in the right-hand side. After that, it remains to estimate inner products. From $\rho$-star-negative comonotonicity we have $- 2\gamma_2 \langle \tx^k - x^*, F(\tx^k) \rangle  \leq 2\rho\gamma_2 \|F(\tx^k)\|^2$. For the second inner product $- 2\gamma_1\gamma_2 \langle F(x^k), F(\tx^k) \rangle$ we use $2\langle a,b \rangle = \|a\|^2 + \|b\|^2 - \|a - b\|^2$, which holds for any $a,b \in \R^d$, and then apply $L$-Lipschitzness to upper bound the term $\gamma_1\gamma_2\|F(x^k) - F(\tx^k)\|^2$. Finally, we rearrange the terms, see the full proof in Appendix~\ref{appendix:eg}.
\end{proof}

From the above result one can easily notice that the choice of $\gamma_2 \leq \gamma_1 - 2\rho$ and $\gamma_1 < \nicefrac{1}{L}$ implies best-iterate convergence in terms of the squared norm of the operator. However, in this proof, $\gamma_2$ should be positive, i.e., this proof is valid only for $\gamma_1 > 2\rho$. In other words, one can derive best-iterate $\cO(\nicefrac{1}{N})$ rate for \ref{eq:EG} whenever $\rho < \nicefrac{1}{2L}$, which is also known from \citet{pethick2022escaping} (though \citet{pethick2022escaping} do not provide analogs of Lemma~\ref{lem:EG_main_lemma}).

Next, to get the last-iterate convergence of \ref{eq:EG} we assume $\rho$-negative comonotonicity, since even for \ref{eq:PP} -- a simpler algorithm -- we need to do this. Moreover, even in the monotone case the existing proofs of the last-iterate convergence of \ref{eq:EG} rely on the usage of same stepsizes $\gamma_1 = \gamma_2 = \gamma$ \citep{gorbunov2021extragradient, cai2022tight}. This partially can be explained by the following fact: $\|F(x^{k+1})\|$ can be larger than $\|F(x^k)\|$ if $\gamma_1 \neq \gamma_2$ \citep{gorbunov2021extragradient}. Therefore, we also assume that $\gamma_1 = \gamma_2 = \gamma$ to derive last-iterate convergence rate.

However, as Lemma~\ref{lem:EG_main_lemma} indicates, the choice $\gamma_1 = \gamma_2 = \gamma$ may complicate the proof because the term from \eqref{eq:EG_key_inequality_main_1} becomes non-negative. Moreover, it is natural to expect that the proof will work for smaller range of $\rho$. Nevertheless, using computer-assisted approach, we derive that for any $\rho \leq \nicefrac{1}{8L}$ and $4\rho\leq \gamma \leq \nicefrac{1}{2L}$ \ref{eq:EG} the iterates of \ref{eq:EG} satisfy $\|F(x^{k+1})\| \leq \|F(x^k)\|$ which is the main building block of the obtained proof.

We summarized the derived upper-bounds for \ref{eq:EG} in the following result.

\begin{theorem}\label{thm:EG_convergence}
    Let $F$ be $L$-Lipschitz and $\rho$-star-negative comonotone with $\rho < \nicefrac{1}{2L}$. Then, for any $2\rho < \gamma_1 < \nicefrac{1}{L}$ and $0 < \gamma_2 \leq \gamma_1 - 2\rho$ the iterates produced by \ref{eq:EG} after $N\geq 0$ iteration satisfy
    \begin{equation}
        \frac{1}{N+1}\sum\limits_{k=0}^N \|F(x^k)\|^2 \leq \frac{\|x^0 - x^*\|^2}{\gamma_1\gamma_2(1 - L^2\gamma_1^2)(N+1)}. \label{eq:EG_best_iterate}
    \end{equation}
    If, in addition, $F$ is $\rho$-negative comonotone with $\rho \leq \nicefrac{1}{8L}$ and $\gamma_1 = \gamma_2 = \gamma$ such that $4\rho \leq \gamma \leq \nicefrac{1}{2L}$, then for any $k \geq 0$ the iterates produced by \ref{eq:EG} satisfy $\|F(x^{k+1})\| \leq \|F(x^k)\|$ and for any $N\geq 1$
    \begin{equation}
        \|F(x^N)\|^2 \leq \frac{28 \|x^0 - x^*\|^2}{N\gamma^2 + 320\gamma \rho}. \label{eq:EG_last_iterate}
    \end{equation}
\end{theorem}

The results similar to \eqref{eq:EG_best_iterate} are known in the literature: \citet{diakonikolas2021efficient} derives best-iterate $\cO(\nicefrac{1}{N})$ convergence for $\rho < \nicefrac{1}{8L}$ and \citet{pethick2022escaping} generalizes it to the case of any $\rho < \nicefrac{1}{2L}$. In this sense, \eqref{eq:EG_best_iterate} recovers the one from \citet{pethick2022escaping}, though the proof is different.

Next, the last-iterate convergence result from \eqref{eq:EG_last_iterate} holds for any $\rho \leq \nicefrac{1}{8L}$, which is much smaller than the range $\rho < \nicefrac{1}{2L}$ allowed for the best-iterate result. Nevertheless, the previous best-known last-iterate rate requires $\rho$ to be smaller than $\nicefrac{1}{16L}$ \citep{luo2022last}, which is $2$ times smaller than what is allowed for \eqref{eq:EG_last_iterate}.

This discussion naturally leads us to the following question: \emph{for given $L > 0$ what is the maximal possible $\rho$ for which there exists a choice of stepsizes in \ref{eq:EG} such that it converges for any $\rho$-negative comonotone $L$-Lipschitz operator $F$?} This question is partially addressed by \citet{pethick2022escaping}, who prove that if $\gamma_1 = \nicefrac{1}{L}$, then for $\rho \geq \nicefrac{(1 - L\gamma_2)}{2L}$ \ref{eq:EG} does not necessary converge. Guided by the results obtained for \ref{eq:PP}, we make a further step and derive the following statement.

\begin{theorem}\label{thm:EG_counter_example}
    For any $L > 0$, $\rho \geq \nicefrac{1}{2L}$, and any choice of stepsizes $\gamma_1, \gamma_2 > 0$ there exists $\rho$-negative comonotone $L$-Lipschitz operator $F$ such that \ref{eq:EG} does not necessary converges on solving \ref{eq:VIP} with this operator $F$. In particular, for $\gamma_1 > \nicefrac{1}{L}$ it is sufficient to take $F(x) = L x$, and for $0 < \gamma_1 \leq \nicefrac{1}{L}$ one can take $F(x) = L A x$, where $x \in \R^2$, $$A = \begin{pmatrix} \cos\theta & -\sin\theta \\ \sin\theta & \cos\theta \end{pmatrix},\quad \theta = \frac{2\pi}{3}.$$ 
\end{theorem}

This result corroborates Theorem~\ref{thm:PP_counter_example} and known relationship between \ref{eq:EG} and \ref{eq:PP}. That is, from the one side, it is known that \ref{eq:EG} can be seen as an approximation of \ref{eq:PP} \citep[Theorem 1]{mishchenko2020revisiting}. Since for \ref{eq:PP} converges only for the stepsizes larger than $2\rho$, it is natural to expect that \ref{eq:EG} also needs to have at least one stepsize larger than $2\rho$ (otherwise, it can be seen as an approximation of \ref{eq:PP} with stepsize not larger than $2\rho$ that is known to be non-convergent). From the other side, unlike \ref{eq:PP}, \ref{eq:EG} does not converge for arbitrary large stepsizes, which is a standard phenomenon for explicit methods in optimization. In particular, one has to take $\gamma_1 \leq \nicefrac{1}{L}$ (otherwise there exists a ``very good'' -- $L$-cocoercive -- operator such that \ref{eq:EG} diverges). These two observations explain the intuition behind Theorem~\ref{thm:EG_counter_example}.

\paragraph{Optimistic gradient.} Optimistic gradient \citep{popov1980modification} is a single-call version of \ref{eq:EG} having the following form:
\begin{equation}\tag{\algname{OG}}\label{eq:OG}
    \begin{aligned}
    \tx^{k} &= x^k  - \gamma_1 F(\tx^{k-1}),\quad \forall k > 0,\\
    x^{k+1} &= x^k - \gamma_2 F(\tx^k),\quad \forall k \geq 0,
    \end{aligned}
\end{equation}
where $\tx^0 = x^0$. Guided by the results and intuition developed for \ref{eq:EG}, here we also deviate from the original form of \ref{eq:OG}, which has $\gamma_1 = \gamma_2$, and allow $\gamma_1$ and $\gamma_2$ being different. The main goal of the rest of this section is in the obtaining the results on the convergence of \ref{eq:OG} similar to what are derived for \ref{eq:EG} earlier in this section.

Before we move on, we would like to highlight the challenges in the analysis of \ref{eq:OG}. Although \ref{eq:EG} and \ref{eq:OG} can both be seen as approximations of \ref{eq:PP} \citep{mokhtari2020unified}, they have some noticeable theoretical differences going beyond algorithmic ones. For example, even for monotone $L$-Lipschitz operator $F$ the iterates produced by \ref{eq:OG} do not satisfy $\|F(x^{k+1})\| \leq \|F(x^k)\|$ or $\|F(\tx^k)\| \leq \|F(x^k)\|$ in general \citep{gorbunov2022last}, while for \ref{eq:EG} $\|F(x^{k+1})\| \leq \|F(x^k)\|$ holds \citep{gorbunov2021extragradient}. This fact makes the analysis of \ref{eq:OG} more complicated than in the case of \ref{eq:EG}. Moreover, the known convergence results in the monotone case for \ref{eq:OG} require smaller stepsizes than for \ref{eq:EG} \citep{gorbunov2022last, cai2022tight}. In view of the obtained results for \ref{eq:PP} and \ref{eq:EG}, this fact highlights non-triviality of obtaining convergence results for \ref{eq:OG} under $\rho$-negative comonotonicity for the same range of allowed values $\rho$ as for \ref{eq:EG}.

Nevertheless, we obtain the best-iterate $\cO(\nicefrac{1}{N})$ convergence of \ref{eq:OG} for any $\rho < \nicefrac{1}{2L}$, i.e., for the same range of $\rho$ as for \ref{eq:EG}. We also derive last-iterate $\cO(\nicefrac{1}{N})$ convergence of \ref{eq:OG} but for $\rho \leq \nicefrac{5}{62L}$, which is a smaller range than we have for \ref{eq:EG}. The results are summarized below.

\begin{theorem}\label{thm:OG_convergence}
    Let $F$ be $L$-Lipschitz and $\rho$-star-negative comonotone with $\rho < \nicefrac{1}{2L}$. Then, for any $2\rho < \gamma_1 < \nicefrac{1}{L}$ and $0 < \gamma_2 \leq \min\{\nicefrac{1}{L} - \gamma_1, \gamma_1 - 2\rho\}$ the iterates produced by \ref{eq:OG} after $N\geq 0$ iteration satisfy
    \begin{equation}
        \frac{1}{N+1}\sum\limits_{k=0}^N \|F(x^k)\|^2 \leq \frac{\|x^0 - x^*\|^2}{\gamma_1\gamma_2(1 - L^2(\gamma_1 + \gamma_2)^2)(N+1)}.
        \label{eq:OG_best_iterate}
    \end{equation}
    If, in addition, $F$ is $\rho$-negative comonotone with $\rho \leq \nicefrac{5}{62L}$ and $\gamma_1 = \gamma_2 = \gamma$ such that $4\rho \leq \gamma \leq \nicefrac{10}{31L}$, then for any $N \geq 1$ the iterates produced by \ref{eq:OG} satisfy
    \begin{equation}
        \|F(x^N)\|^2 \leq \frac{717\|x^0 - x^*\|^2}{N \gamma(\gamma - 3\rho) + 800\gamma^2}. \label{eq:OG_last_iterate}
    \end{equation}
\end{theorem}

The derived best-iterate rate \eqref{eq:OG_best_iterate} for \ref{eq:OG} is not new: \citet{bohm2022solving} proves a similar result for the same range of $\rho$, though the proof that we provide differs from the proof by \citet{bohm2022solving}. Similarly to the case of \ref{eq:EG}, it is valid for any $\rho < \nicefrac{1}{2L}$. Next, the last-iterate $\cO(\nicefrac{1}{N})$ rate is recently obtained for \ref{eq:OG} by \citet{luo2022last}. It holds for any $\rho < \nicefrac{8}{(27\sqrt{6}L)}$, while the rate that we obtain is valid for any $\rho \leq \nicefrac{5}{62L}$, which is $\approx 1.33$ times larger range.

Finally, as for \ref{eq:EG}, we derive the following result about the largest possible range for $\rho$ in the case of \ref{eq:OG}.

\begin{theorem}\label{thm:OG_counter_example}
    For any $L > 0$, $\rho \geq \nicefrac{1}{2L}$, and any choice of stepsizes $\gamma_1, \gamma_2 > 0$ there exists $\rho$-negative comonotone $L$-Lipschitz operator $F$ such that \ref{eq:OG} does not necessary converges on solving \ref{eq:VIP} with this operator $F$. In particular, for $\gamma_1 > \nicefrac{1}{L}$ it is sufficient to take $F(x) = L x$, and for $0 < \gamma_1 \leq \nicefrac{1}{L}$ one can take $F(x) = L A x$, where $x \in \R^2$, $$A = \begin{pmatrix} \cos\theta & -\sin\theta \\ \sin\theta & \cos\theta \end{pmatrix},\quad \theta = \frac{2\pi}{3}.$$ 
\end{theorem}

Note that the counter-examples are exactly the same as for \ref{eq:EG}. Moreover, since \ref{eq:OG} can be seen as an approximation of \ref{eq:PP}, this result is expected and the same has the same intuition behind as Theorem~\ref{thm:EG_counter_example}.

\section{Discussion}

In this work, we studied worst-case convergence of methods for solving \ref{eq:MI}/\ref{eq:VIP} with (star-)negative-comonotone operators, which we believe is an important first step for going beyond the very popular monotonocity assumption, that is often not satisfied in modern applications.

Namely, we study the proximal point~\eqref{eq:PP}, the extragradient~\eqref{eq:EG}, and the optimistic gradient~\eqref{eq:OG} methods. Although the basic understanding of the convergence of \ref{eq:PP} and best-iterate convergence of \ref{eq:EG} and \ref{eq:OG} is relatively complete, several open-questions about last-iterate convergence of \ref{eq:EG} and \ref{eq:OG} remain. In particular, it is unclear what is the largest possible range for $\rho$ for which one can guarantee last-iterate $\cO(\nicefrac{1}{N})$ convergence of  \ref{eq:EG}/\ref{eq:OG} under $\rho$-negative comonotonicity and $L$-Lipschitzness.

Moreover, another important direction for future research is identifying weaker assumptions allowing to prove non-asymptotic convergence rates for \ref{eq:PP}/\ref{eq:EG}/\ref{eq:OG} and at the same time allowing to have isolated optima or non-convex solution sets, as discussed in Section~\ref{sec:neg_comon}. Finally, it would be very important to extend the results to the stochastic case; see \citep{pethick2023solving} for the recent advances in this direction.

\section*{Acknowledgements}

We thank Axel B\"ohm for pointing us to the reference \citep{bohm2022solving}, which we were not aware of during the work on the paper. We also thank an anonymous reviewer for pointing us to the references \citep{bauschke2021generalized, iusem2003inexact}, which we were also not aware of during the work on the paper. The research of E.~Gorbunov was partially supported by a grant for research centers in the field of artificial intelligence, provided by the Analytical Center for the Government of the Russian Federation in accordance with the subsidy agreement (agreement identifier 000000D730321P5Q0002) and the agreement with the Moscow Institute of Physics and Technology dated November 1, 2021 No. 70-2021-00138.

\bibliography{refs}
\bibliographystyle{icml2023}

\newpage
\appendix
\onecolumn
{\small\tableofcontents}


\clearpage

\section{Missing Proofs and Details From Section~\ref{sec:neg_comon}}

\begin{lemma}[Lemma~\ref{lem:expansiveness_of_neg_comon_operator}]
    $F:\R^d \rightrightarrows \R^d$ is $\rho$-negative comonotone ($\rho\geq 0$) if and only if operator $\Id + 2\rho F$ is expansive.
\end{lemma}
\begin{proof}
    Expansiveness of operator $\Id + 2\rho F$ means that for any $x,y \in \R^d$
    \begin{eqnarray*}
        \|x + 2\rho F(x) - y - 2\rho F(y)\|^2 \geq \|x - y\|^2,
    \end{eqnarray*}
    where $F(x)$ and $F(y)$ represent the arbitrary elements from the values of $F$ measured at $x$ and $y$, respectively. Expanding the square in the left-hand side of the above inequality, we get
    \begin{eqnarray*}
        \|x - y\|^2 + 4\rho\langle F(x) - F(y), x - y \rangle + 4\rho^2 \|F(x) - F(y)\|^2 \geq \|x-y\|^2.
    \end{eqnarray*}
    The above inequality is equivalent to $\rho$-negative comonotonicity \eqref{eq:rho_neg_comon}.
\end{proof}

\begin{theorem}[Theorem~\ref{thm:spectral_viewpoint_on_neg_comon}]
    Let $F:\R^d \to \R^d$ be a continuously differentiable. Then, the following statements are equivalent:
    \begin{itemize}
        \item $F$ is $\rho$-negative comonotone,
        \item $\Re(\nicefrac{1}{\lambda}) \geq -\rho$ for all $\lambda \in \Sp(\nabla F(x))$, $\forall x \in \R^d$.
    \end{itemize}
\end{theorem}
\begin{proof}
    For a complex number $\lambda$ condition $\Re(\nicefrac{1}{\lambda}) \geq -\rho$ is equivalent to $|\lambda + \nicefrac{1}{2\rho}| \geq \nicefrac{1}{2\rho}$. Indeed, for $\lambda = \lambda_1 + i\lambda_2$, $\lambda_1,\lambda_2 \in \CC$ we have
    \begin{equation}
        \Re\left(\frac{1}{\lambda}\right) = \frac{\lambda_1}{\lambda_1^2 + \lambda_2^2} \geq -\rho\quad \Longleftrightarrow \quad \lambda_1^2 + \lambda_2^2 + \frac{\lambda_1}{\rho} \geq 0 \quad \Longleftrightarrow \quad \left|\lambda + \frac{1}{2\rho} \right| \geq \frac{1}{2\rho}, \label{eq:csjjbhscjbdh}
    \end{equation}
    i.e., $\Re(\nicefrac{1}{\lambda}) \geq -\rho$ means that $\lambda$ lies in the disc in $\CC$ centered at $(-\nicefrac{1}{2\rho})$ and radius $\nicefrac{1}{2\rho}$. From the other side, $\rho$-negative comonotonicity of $F$ is equivalent to expansiveness of $\Id + 2\rho F$ (Lemma~\ref{lem:expansiveness_of_neg_comon_operator}), which is equivalent to $|\lambda| \geq 1$ for any $\lambda \in \Sp(I + 2\rho\nabla F(x))$ and for any $x\in\R^d$. Since
    \begin{equation}
        \Sp(I + 2\rho\nabla F(x)) = \left\{1 + 2\rho\lambda\mid \lambda \in \Sp(\nabla F(x))\right\}, \notag
    \end{equation}
    we get that $|1 + 2\rho \lambda| \geq 1$ for any $\lambda \in \Sp(\nabla F(x))$ and any $x\in \R^d$. Taking into account \eqref{eq:csjjbhscjbdh}, we obtain the desired result.
\end{proof}

\newpage

\section{Missing Details on PEP From Section~\ref{sec:prox_point}}\label{appendix:PEP}

\paragraph{On PEP formulation \eqref{eq:PEP_for_PP}.} To find the tight convergence rate of \ref{eq:PP} and build worst-case examples of $\rho$-negative comonotone operators for \ref{eq:PP}, we consider problem \eqref{eq:PEP_for_PP}, which we restate below for convenience:
\begin{eqnarray}
    \max\limits_{F, d, x^0}&&\|x^N - x^{N-1}\|^2 \label{eq:PEP_for_PP_appendix}\\
    \text{s.t.}&& F:\R^d \rightrightarrows \R^d \text{ is $\rho$-negative comonotone},\notag\\
    &&\|x^0 - x^*\|^2 \leq R^2,\; 0 \in  F(x^*),\notag\\
    && x^{k+1} = x^k - \gamma F(x^{k+1}),\quad  k=0,1,\ldots,N-1.\notag
\end{eqnarray}
The above problem requires maximization over \emph{infinitely-dimensional} space of $\rho$-negative comonotone operators. To solve such a problem numerically, one can properly reformulate it to a \emph{finite-dimensional} one. Whereas PEPs were introduced by~\citet{drori2012performance}, thie reformulation technique was provided in~\citet{taylor2017smooth,2017taylor} in the context of optimization problems, and was extended to problems involving (monotone) operators in~\citet{ryu2020operator}. In particular, instead of \eqref{eq:PEP_for_PP_appendix}, one can consider an equivalent finite-dimensional problem
\begin{eqnarray}
    \max\limits_{\substack{d\\x^*,x^0, x^1, \ldots, x^N \in \R^d\\ g^*,g^0, g^1, \ldots, g^N \in \R^d}}&&\|x^N - x^{N-1}\|^2 \label{eq:PEP_for_PP_appendix_1}\\
    \text{s.t.}&& F:\R^d \rightrightarrows \R^d \text{ is $\rho$-negative comonotone}, \label{eq:non_trivial_constraint_1}\\
    && g^k \in F(x^k),\quad  k=*,0,1,\ldots,N,\quad g^* = 0, \label{eq:non_trivial_constraint_2}\\
    &&\|x^0 - x^*\|^2 \leq R^2,\notag\\
    && x^{k+1} = x^k - \gamma g^{k+1},\quad  k=0,1,\ldots,N-1.\notag
\end{eqnarray}
Although the above problem is finite-dimensional, it has non-trivial constraints~\eqref{eq:non_trivial_constraint_1}-\eqref{eq:non_trivial_constraint_2}, which can be handled via the following result.

\begin{theorem}\label{thm:tho_comon_interpolation}
    Let $\{(x^k, g^k)\}_{k=0}^N \subseteq \R^d \times \R^d$ be some finite set of pairs of points in $\R^d$. There exists a maximal $\rho$-negative comonotone operator $F:\R^d \rightrightarrows \R^d$ such that $g^k \in F(x^k)$, $k = 0,\ldots, N$ if and only if
    \begin{equation}
        \langle g^i - g^j, x^i - x^j \rangle \geq -\rho\|g^i - g^j\|^2\quad \forall i,j =  0,\ldots, N.\label{eq:interpolation_conditions} 
    \end{equation}
\end{theorem}
\begin{proof}
    Following~\citet{ryu2020operator}, we say that the set $\{(x^k, g^k)\}_{k=0}^N \subseteq \R^d \times \R^d$ is $\cM$-interpolable if there exists a maximal monotone ($0$-negative comonotone) operator $F:\R^d \rightrightarrows \R^d$ such that $g^k \in F(x^k)$, $k = 0,\ldots, N$. One can introduce a similar notion for $\rho$-negative comonotone case, i.e., we say that the set $\{(x^k, g^k)\}_{k=0}^N \subseteq \R^d \times \R^d$ is $\cN\cM_\rho$-interpolable if there exists a maximal $\rho$-negative comonotone operator $F:\R^d \rightrightarrows \R^d$ such that $g^k \in F(x^k)$, $k = 0,\ldots, N$. Next, for convenience we denote the classes of maximal monotone and maximal $\rho$-negative comonotone operators as $\cM$ and $\cN\cM_\rho$ respectively. Then, in view of the maximal monotone extension theorem~\citep[Theorem 20.21]{bauschke2011convex}, the set $\{(x^k, g^k)\}_{k=0}^N \subseteq \R^d \times \R^d$ is $\cM$-interpolable if and only if $\langle g^i - g^j, x^i - x^j \rangle \geq 0$ for any $i,j =  0,\ldots, N$. 

    For obtaining the desired result, we simply reduce the problem of finding a maximal $\rho$-negative comonotone interpolating operator for the set $\{(x^k,g^k)\}_{k=0}^N$ as that of finding a maximal monotone operator interpolating $\{(x^k+\rho g^k, g^k)\}_{k=0}^N$, which is a consequence of the following equivalence: an operator $F:\mathbb{R}^d\rightrightarrows \mathbb{R}$ is maximal $\rho$-negatively monotone if and only if $F^{-1}+\rho \Id$ is maximal monotone. More precisely, the reasoning is as follows:
    \begin{eqnarray*}
        \langle g^i - g^j, x^i - x^j \rangle &\geq &-\rho\|g^i - g^j\|^2 \quad \forall i,j =  0,\ldots, N\\
        & \Longleftrightarrow &\; \langle g^i - g^j, x^i + \rho g^i - (x^j + \rho g^j) \rangle \geq 0 \quad \forall i,j =  0,\ldots, N\\
        & \Longleftrightarrow& \; \exists\; T \in \cM:\; x^i + \rho g^i \in T(g^i) \quad \forall i =  0,\ldots, N\\
        & \overset{Q = T - \rho\Id}{\Longleftrightarrow}& \; \exists\; Q:\; Q + \rho \Id \in \cM \text{ and } x^i \in Q(g^i) \quad \forall i =  0,\ldots, N\\
        & \overset{F = Q^{-1}}{\Longleftrightarrow}& \; \exists\; F \in \cN\cM_{\rho} \text{ and } g^i \in F(x^i) \quad \forall i =  0,\ldots, N,
    \end{eqnarray*}
    where the last equivalence follows from the following fact: $F$ is $\rho$-negative comonotone if and only if $F^{-1} + \rho\Id$ is monotone, thereby concluding the proof.
\end{proof}

In view of the above theorem, one can replace \eqref{eq:non_trivial_constraint_1}-\eqref{eq:non_trivial_constraint_2} constraints by $(N+1)(N+2)$ inequalities of the type \eqref{eq:interpolation_conditions} and get the following finite-dimensional problem, which is equivalent to \eqref{eq:PEP_for_PP_appendix_1}:
\begin{eqnarray}
    \max\limits_{\substack{d\\x^*,x^0, x^1, \ldots, x^N \in \R^d\\ g^*,g^0, g^1, \ldots, g^N \in \R^d}}&&\|x^N - x^{N-1}\|^2 \label{eq:PEP_for_PP_appendix_2}\\
    \text{s.t.}&& \langle g^i - g^j, x^i - x^j \rangle \geq -\rho\|g^i - g^j\|^2,\quad  i,j=*,0,1,\ldots,N,\quad g^* = 0, \label{eq:simple_constraints}\\
    &&\|x^0 - x^*\|^2 \leq R^2,\label{eq:init_constraint}\\
    && x^{k+1} = x^k - \gamma g^{k+1},\quad  k=0,1,\ldots,N-1.\notag
\end{eqnarray}
We notice that $x^1,\ldots, x^N$ are linear combinations of $x^0, g^0, g^1, \ldots, g^N$ and we also have constraint $g^* = 0$. Therefore, one can reduce the number of maximization vector-variables to $N+3$: $x^*, x^0, g^0, g^1, \ldots, g^N$. Moreover, the above problem is linear w.r.t.\ the inner products of all possible pairs of vectors $x^*, x^0, g^0, g^1, \ldots, g^N$. This means that \eqref{eq:PEP_for_PP_appendix_2} is linear w.r.t.\ the elements of matrix $G = V^\top V$, where $V = (x^*, x^0, g^0, g^1, \ldots, g^N)$, and one can reformulate the problem \eqref{eq:PEP_for_PP_appendix_2} as the following semidefinite programming (SDP)
\begin{eqnarray}
    \max\limits_{G \in \mathbb{S}_{+}^{N+3}}&& \Tr(M_0 G) \label{eq:PEP_for_PP_SDP}\\
    \text{s.t.}&& \Tr(M_i G) \leq 0,\quad  i=1,2\ldots,(N+2)(N+3), \notag\\
    && \Tr(M_{-1} G) \leq R^2.\notag
\end{eqnarray}
Here $\mathbb{S}_{+}^{N+3}$ denotes the set of symmetric positive semidefinite matrices of size $(N+3)\times(N+3)$ and matrices $M_0$, $\{M_i\}_{i=1}^{(N+2)(N+3)}$, and $M_{-1}$ encode the objective \eqref{eq:PEP_for_PP_appendix_2} and constraints \eqref{eq:simple_constraints}-\eqref{eq:init_constraint}, respectively. We do not provide the exact formulas for these matrices and refer to the examples of how they can be constructed provided in \citep{ryu2020operator, gorbunov2021extragradient}. We note that in toolboxes like PESTO \citep{taylor2017performance} and PEPit \citep{goujaud2022pepit}, the process of constructing matrices $M_0$, $\{M_i\}_{i=1}^{(N+2)(N+3)}$, $M_{-1}$ is fully automated.

\paragraph{On low-dimensional worst-case examples.} It is worth mentioning that for any $G \in \mathbb{S}_{+}^{N+3}$ one can reconstruct vectors $x^*, x^0, g^0, g^1, \ldots, g^N \in \R^{N+3}$ such that $G$ is their Gram matrix, i.e., find $V = (x^*, x^0, g^0, g^1, \ldots, g^N) \in \R^{(N+3)\times (N+3)}$ such that $G = V^\top V$. More precisely, if $\mathrm{rank}(G) = r \leq N+3$, then one can find $x^*, x^0, g^0, g^1, \ldots, g^N \in \R^{r}$ such that $G$ is the Gram matrix of this set of vectors.

Therefore, to obtain low-dimensional worst-case trajectories like ones illustrated in Figure~\ref{fig:1b}, we need to find low-rank solution of \eqref{eq:PEP_for_PP_SDP}. To do so, we apply \emph{trace heuristic} \citep{taylor2017performance}, where we first find numerically an approximate optimal value $v_*$ of problem \eqref{eq:PEP_for_PP_SDP} and then solve the following problem:
\begin{eqnarray}
    \min\limits_{G \in \mathbb{S}_{+}^{N+3}}&& \Tr(G)  \label{eq:PEP_for_PP_SDP_trace_heuristic}\\
    \text{s.t.}&& \Tr(M_i G)\leq 0,\quad  i=1,2\ldots,(N+2)(N+3), \notag\\
    && \Tr(M_{-1} G) \leq R^2,\notag\\
    && \Tr(M_0 G) = v_*. \label{eq:new_constr}
\end{eqnarray}
Constraint \eqref{eq:new_constr} enforces that by solving the above problem we find numerically an approximate solution for \eqref{eq:PEP_for_PP_SDP} of a comparable quality and minimization of $\Tr(G)$ can be seen as an ``approximate minimization'' of $\mathrm{rank}(G)$.

\newpage

\section{Missing Proofs and Details From Section~\ref{sec:eg_peg}}

This appendix provides the complete proofs of the results of~\ref{eq:EG} and~\ref{eq:OG}.

\subsection{Extragradient method}\label{appendix:eg}

\subsubsection{Guarantees for the averaged squared norm of the operator}

Theorem~\ref{thm:EG_convergence} consists of the two results: one requires only star negative comonotonicity and gives the rate in terms of the averaged squared norms of the operator along the trajectory and the other one requires negative comonotonicity but gives last-iterate convergence guarantee. We start with the first result which is a simplification of Theorem 3.1 from \citet{pethick2022escaping}. Our proof is a bit more explicit in terms of why we need $\gamma_1$ to be large, because it relies on the following lemma.
\begin{lemma}\label{lem:EG_main_lemma_appendix}
    Let $F$ be $L$-Lipschitz and $\rho$-star-negative comonotone. Then, for any $k\geq 0$ the iterates produced by \ref{eq:EG} after $k\geq 0$ iterations satisfy
    \begin{equation}
        \|x^{k+1} - x^*\|^2 \leq \|x^k - x^*\|^2 - \gamma_2\left(\gamma_1 - 2\rho - \gamma_2\right)\|F(\tx^k)\|^2 - \gamma_1\gamma_2(1 - L^2\gamma_1^2)\|F(x^k)\|^2. \label{eq:EG_key_inequality}
    \end{equation}
\end{lemma}
\begin{proof}
    By the definition of $x^{k+1}$ and $\tx^k$ we have
    \begin{eqnarray*}
        \|x^{k+1} - x^*\|^2 &=& \|x^k - x^*\|^2 - 2\gamma_2 \langle x^k - x^*, F(\tx^k) \rangle + \gamma_2^2 \|F(\tx^k)\|^2\\
        &=& \|x^k - x^*\|^2 - 2\gamma_2 \langle \tx^k - x^*, F(\tx^k) \rangle - 2\gamma_1\gamma_2 \langle F(x^k), F(\tx^k) \rangle + \gamma_2^2 \|F(\tx^k)\|^2.
    \end{eqnarray*}
    Next, we estimate the second term in the right-hand side using star-negative comonotonicity:
    \begin{eqnarray*}
        \|x^{k+1} - x^*\|^2 &\overset{\eqref{eq:rho_star_neg_comon}}{\leq}& \|x^k - x^*\|^2 + \gamma_2\left(2\rho + \gamma_2\right)\|F(\tx^k)\|^2 - 2\gamma_1\gamma_2 \langle F(x^k), F(\tx^k) \rangle.
    \end{eqnarray*}
    Finally, we handle the last term in the right-hand side of the above inequality using $2\langle a,b \rangle = \|a\|^2 + \|b\|^2 - \|a - b\|^2$, which holds for any $a,b \in \R^d$, and then applying $L$-Lipschitzness of $F$:
    \begin{eqnarray}
        \|x^{k+1} - x^*\|^2 &\leq& \|x^k - x^*\|^2 - \gamma_2\left(\gamma_1 - 2\rho - \gamma_2\right)\|F(\tx^k)\|^2 - \gamma_1\gamma_2\|F(x^k)\|^2\notag\\
        &&\quad + \gamma_1\gamma_2\|F(x^k) - F(\tx^k)\|^2 \notag\\
        &\overset{\eqref{eq:Lipschitzness}}{\leq}& \|x^k - x^*\|^2 - \gamma_2\left(\gamma_1 - 2\rho - \gamma_2\right)\|F(\tx^k)\|^2 - \gamma_1\gamma_2\|F(x^k)\|^2\notag\\
        &&\quad + \gamma_1\gamma_2L^2\|x^k - \tx^k\|^2. \notag
    \end{eqnarray}
    Taking into account $x^k - \tx^k = \gamma_1 F(x^k)$ and rearranging the terms, we get the result.
\end{proof}

This lemma implies the first part of Theorem~\ref{thm:EG_convergence} and even a bit more.
\begin{theorem}[First part of Theorem~\ref{thm:EG_convergence}]
    Let $F$ be $L$-Lipschitz and $\rho$-star-negative comonotone with $\rho < \nicefrac{1}{2L}$. If $2\rho < \gamma_1 < \nicefrac{1}{L}$ and $0 < \gamma_2 \leq \gamma_1 - 2\rho$, then the iterates produced by \ref{eq:EG} after $N\geq 0$ iteration satisfy
    \begin{equation}
        \frac{1}{N+1}\sum\limits_{k=0}^N \|F(x^k)\|^2 \leq \frac{\|x^0 - x^*\|^2}{\gamma_1\gamma_2(1 - L^2\gamma_1^2)(N+1)}. \label{eq:EG_best_iterate_appendix_1}
    \end{equation}
    If $2\rho < \gamma_1 \leq \nicefrac{1}{L}$ and $0 < \gamma_2 < \gamma_1 - 2\rho$, then the iterates produced by \ref{eq:EG} after $N\geq 0$ iteration satisfy
    \begin{equation}
        \frac{1}{N+1}\sum\limits_{k=0}^N \|F(\tx^k)\|^2 \leq \frac{\|x^0 - x^*\|^2}{\gamma_2(\gamma_1 - 2\rho - \gamma_2)(N+1)}. \label{eq:EG_best_iterate_appendix_2}
    \end{equation}
\end{theorem}
\begin{proof}
    First, we consider the case when $2\rho < \gamma_1 < \nicefrac{1}{L}$ and $0 < \gamma_2 \leq \gamma_1 - 2\rho$. In this case, $\gamma_2(\gamma_1 - 2\rho - \gamma_2) \geq 0$ and $\gamma_1\gamma_2(1 - L^2\gamma_1^2) > 0$. Therefore, Lemma~\ref{lem:EG_main_lemma} implies
    \begin{equation*}
        \gamma_1\gamma_2(1 - L^2 \gamma_1^2)\|F(x^k)\|^2 \leq \|x^k - x^*\|^2 - \|x^{k+1} - x^*\|^2.
    \end{equation*}
    Summing up the above inequality for $k = 0,\ldots, N$, dividing the result by $\gamma_1\gamma_2(1 - L^2 \gamma_1^2)(N+1)$, and using $-\|x^{N+1} - x^*\|^2 \leq 0$, we get \eqref{eq:EG_best_iterate_appendix_1}.
    
    Next, we consider the case when $2\rho < \gamma_1 \leq \nicefrac{1}{L}$ and $0 < \gamma_2 < \gamma_1 - 2\rho$. In this case, $\gamma_2(\gamma_1 - 2\rho - \gamma_2) > 0$ and $\gamma_1\gamma_2(1 - L^2\gamma_1^2) \geq 0$. Therefore, Lemma~\ref{lem:EG_main_lemma} implies
    \begin{equation*}
        \gamma_2(\gamma_1 - 2\rho - \gamma_2)\|F(\tx^k)\|^2 \leq \|x^k - x^*\|^2 - \|x^{k+1} - x^*\|^2.
    \end{equation*}
    Summing up the above inequality for $k = 0,\ldots, N$, dividing the result by $\gamma_2(\gamma_1 - 2\rho - \gamma_2)(N+1)$, and using $-\|x^{N+1} - x^*\|^2 \leq 0$, we get \eqref{eq:EG_best_iterate_appendix_2}.
\end{proof}

\subsubsection{Last-iterate guarantees}

We start with the following lemma:
\begin{lemma}\label{lem:EG_norm_decreases}
    Let $F$ be $L$-Lipschitz and $\rho$-negative comonotone. Then for any $k \geq 0$ the iterates produced by \ref{eq:EG} with $\gamma_1 = \gamma_2 = \gamma > 0$ satisfy
    \begin{eqnarray}
        \|F(x^{k+1})\|^2 &\leq& \|F(x^k)\|^2 - \left(\frac{1}{2} - 2L^2\gamma^2\right)\|F(\tx^k) - F(x^k)\|^2 \notag\\
        &&\quad - \left(\frac{1}{2} - \frac{\rho}{\gamma}\right)\|F(\tx^k) - F(x^{k+1})\|^2 - \left(\frac{1}{2} - \frac{2\rho}{\gamma}\right)\|F(x^k) - F(x^{k+1})\|^2. \label{eq:EG_norm_inequality}
    \end{eqnarray}
    If additionally $\gamma \leq \nicefrac{1}{2L}$ and $\gamma \geq 4\rho$, then we have $\|F(x^{k+1})\| \leq \|F(x^k)\|$.
\end{lemma}
\begin{proof}
    From $L$-Lipschitzness and $\rho$-negative comonotonicity of $F$ we have
    \begin{gather*}
        \|F(\tx^k) - F(x^{k+1})\|^2 \leq L^2 \|\tx^k - x^{k+1}\|^2,\\
        \langle F(\tx^k) - F(x^{k+1}), \tx^k - x^{k+1}\rangle \geq -\rho \|F(\tx^k) - F(x^{k+1})\|^2,\\
        \langle F(x^k) - F(x^{k+1}), x^k - x^{k+1} \rangle \geq -\rho \|F(x^k) - F(x^{k+1})\|^2.
    \end{gather*}
    Taking into account $\tx^k - x^{k+1} = \gamma (F(\tx^k) - F(x^k))$ and $x^k - x^{k+1} = \gamma F(\tx^k)$, we get
    \begin{gather*}
        \|F(\tx^k) - F(x^{k+1})\|^2 \leq L^2\gamma^2 \|F(\tx^k) - F(x^k)\|^2,\\
        \gamma\langle F(\tx^k) - F(x^{k+1}), F(\tx^k) - F(x^k)\rangle \geq -\rho \|F(\tx^k) - F(x^{k+1})\|^2,\\
        \gamma\langle F(x^k) - F(x^{k+1}), F(\tx^k) \rangle \geq -\rho \|F(x^k) - F(x^{k+1})\|^2.
    \end{gather*}
    Next, we sum up the above inequalities with weights $2$, $\nicefrac{1}{\gamma}$, and $\nicefrac{2}{\gamma}$ respectively:
    \begin{align*}
        2\|F(\tx^k) - F(x^{k+1})\|^2 &- \frac{\rho}{\gamma}\|F(\tx^k) - F(x^{k+1})\|^2 - \frac{2\rho}{\gamma}\|F(x^k) - F(x^{k+1})\|^2\\
        &\leq 2L^2\gamma^2\|F(\tx^k) - F(x^k)\|^2 + \langle F(\tx^k) - F(x^{k+1}), F(\tx^k) - F(x^k)\rangle\\
        &\quad + 2\langle F(x^k), F(\tx^k) \rangle - 2\langle F(x^{k+1}), F(\tx^k) \rangle.
    \end{align*}
    To get rid of the inner products, we use $2\langle a,b \rangle = \|a\|^2 + \|b\|^2 - \|a - b\|^2$, which holds for any $a,b \in \R^d$. Using this, we continue our derivation as follows:
    \begin{align*}
        2\|F(\tx^k) - F(x^{k+1})\|^2 &- \frac{\rho}{\gamma}\|F(\tx^k) - F(x^{k+1})\|^2 - \frac{2\rho}{\gamma}\|F(x^k) - F(x^{k+1})\|^2\\
        &\leq 2L^2\gamma^2\|F(\tx^k) - F(x^k)\|^2 + \frac{1}{2}\|F(\tx^k) - F(x^{k+1})\|^2 + \frac{1}{2}\|F(\tx^k) - F(x^k)\|^2\\
        &\quad - \frac{1}{2}\|F(x^k) - F(x^{k+1})\|^2 + \|F(x^k)\|^2 + \|F(\tx^k)\|^2 - \|F(\tx^k) - F(x^k)\|^2\\
        &\quad - \|F(x^{k+1})\|^2 - \|F(\tx^k)\|^2 + \|F(\tx^k) - F(x^{k+1})\|^2.
    \end{align*}
    Rearranging the terms we get
    \begin{eqnarray}
        \|F(x^{k+1})\|^2 &\leq& \|F(x^k)\|^2 - \left(\frac{1}{2} - 2L^2\gamma^2\right)\|F(\tx^k) - F(x^k)\|^2 \notag\\
        &&\quad - \left(\frac{1}{2} - \frac{\rho}{\gamma}\right)\|F(\tx^k) - F(x^{k+1})\|^2 - \left(\frac{1}{2} - \frac{2\rho}{\gamma}\right)\|F(x^k) - F(x^{k+1})\|^2, \notag
    \end{eqnarray}
    which concludes the proof.
\end{proof}

Using this lemma we construct the potential-based proof of the last-iterate convergence of \ref{eq:EG}.
\begin{theorem}[Second part of Theorem~\ref{thm:EG_convergence}]
    Let $F$ be $L$-Lipschitz and $\rho$-negative comonotone. Then, for any $k \geq 0$ the iterates produced by \ref{eq:EG} with $\gamma_1 = \gamma_2 = \gamma$ such that $4\rho \leq \gamma \leq \nicefrac{1}{2L}$ satisfy
    \begin{equation}
        \Phi_{k+1} \leq \Phi_k,\quad \text{where}\quad \Phi_k = \|x^k - x^*\|^2 + \left(k\gamma^2\left(1 - \frac{5\rho}{2\gamma} - L^2\gamma^2\right) + 40\gamma\rho\right)\|F(x^k)\|^2. \label{eq:EG_potential}
    \end{equation}
    That is, under the introduced assumptions on $\gamma$ and $\rho$ for any $N \geq 1$ the iterates produced by \ref{eq:EG} satisfy
    \begin{equation}
        \|F(x^N)\|^2 \leq \frac{(1 + 40\gamma\rho L^2)\|x^0 - x^*\|^2}{N\gamma^2\left(1 - \frac{5\rho}{2\gamma} - L^2\gamma^2\right) + 40\gamma\rho}. \label{eq:EG_last_iter_apendix}
    \end{equation}
\end{theorem}
\begin{proof}
    From \eqref{eq:EG_key_inequality} with $\gamma_1 = \gamma_2 = \gamma$ we have
    \begin{equation*}
        \|x^{k+1} - x^*\|^2 \leq \|x^k - x^*\|^2 + 2\gamma\rho\|F(\tx^k)\|^2 - \gamma^2\left(1 - L^2\gamma^2\right) \|F(x^k)\|^2.
    \end{equation*}
    Next, taking into account that $4\rho \leq \gamma \leq \nicefrac{1}{2L}$, we also have from Lemma~\ref{lem:EG_norm_decreases} the following inequality:
    \begin{equation}
        \|F(x^{k+1})\|^2 \leq \|F(x^k)\|^2 - \left(\frac{1}{2} - \frac{\rho}{\gamma}\right)\|F(\tx^k) - F(x^{k+1})\|^2. \label{eq:hcbsdbsdvchgs}
    \end{equation}
    Using these two inequalities, we derive the following upper bound on $\Phi_{k+1}$:
    \begin{eqnarray*}
        \Phi_{k+1} &=& \|x^{k+1} - x^*\|^2 + \left((k+1)\gamma^2\left(1 - \frac{5\rho}{2\gamma} - L^2\gamma^2\right) + 40\gamma\rho\right)\|F(x^{k+1})\|^2\\
        &\leq& \|x^k - x^*\|^2 + 2\gamma\rho\|F(\tx^k)\|^2 - \gamma^2\left(1 - L^2\gamma^2\right) \|F(x^k)\|^2\\
        &&\quad + \left((k+1)\gamma^2\left(1 - \frac{5\rho}{2\gamma} - L^2\gamma^2\right) + 40\gamma\rho\right)\left(\|F(x^k)\|^2 - \left(\frac{1}{2} - \frac{\rho}{\gamma}\right)\|F(\tx^k) - F(x^{k+1})\|^2\right)\\
        &=& \Phi_k + 2\gamma\rho\|F(\tx^k)\|^2 - \frac{5}{2}\gamma\rho\|F(x^k)\|^2\\
        &&\quad - \left((k+1)\gamma^2\left(1 - \frac{5\rho}{2\gamma} - L^2\gamma^2\right) + 40\gamma\rho\right)\left(\frac{1}{2} - \frac{\rho}{\gamma}\right)\|F(\tx^k) - F(x^{k+1})\|^2\\
        &\leq& \Phi_k + 2\gamma\rho\|F(\tx^k)\|^2 - \frac{5}{2}\gamma\rho\|F(x^k)\|^2 - 20\rho(\gamma - 2\rho)\|F(\tx^k) - F(x^{k+1})\|^2.
    \end{eqnarray*}
    Finally, we apply $\|a+b\|^2 \leq (1+\beta)\|a\|^2 + (1+\beta^{-1})\|b\|^2$, which holds $\forall a,b\in \R^d$, $\beta > 0$, with $\beta = \nicefrac{1}{4}$ to upper bound the second term in the right-hand side of the above inequality and continue our derivation as follows:
    \begin{eqnarray*}
        \Phi_{k+1} &\leq& \Phi_k + 2\gamma\rho\|F(x^{k+1}) + F(\tx^k) - F(x^{k+1})\|^2 - \frac{5}{2}\gamma\rho\|F(x^k)\|^2 \\
        &&\quad - 20\rho(\gamma - 2\rho)\|F(\tx^k) - F(x^{k+1})\|^2\\
        &\leq& \Phi_k + 2\gamma\rho\left(1 + \frac{1}{4}\right)\|F(x^{k+1})\|^2 + 2\gamma\rho\left(1 + 4\right)\|F(\tx^k) - F(x^{k+1})\|^2 - \frac{5}{2}\gamma\rho\|F(x^k)\|^2\\
        &&\quad - 20\rho(\gamma - 2\rho)\|F(\tx^k) - F(x^{k+1})\|^2\\
        &=& \Phi_k + \frac{5}{2}\gamma\rho\|F(x^{k+1})\|^2 - \frac{5}{2}\gamma\rho\|F(x^k)\|^2 - 10\rho\left(\gamma - 4\rho\right)\|F(\tx^k) - F(x^{k+1})\|^2.
    \end{eqnarray*}
    Taking into account $\|F(x^{k+1})\|^2 \overset{\eqref{eq:hcbsdbsdvchgs}}{\leq} \|F(x^k)\|^2$ and $\gamma \geq 4\rho$, we get \eqref{eq:EG_potential}. Next, we unroll \eqref{eq:EG_potential} and derive \eqref{eq:EG_last_iter_apendix}:
    \begin{eqnarray*}
        \|F(x^N)\|^2 &\leq& \frac{1}{N\gamma^2\left(1 - \frac{5\rho}{2\gamma} - L^2\gamma^2\right) + 40\gamma\rho}\Phi_N \\
        &\leq& \frac{1}{N\gamma^2\left(1 - \frac{5\rho}{2\gamma} - L^2\gamma^2\right) + 40\gamma\rho}\Phi_{N-1} \leq \ldots \leq \frac{1}{N\gamma^2\left(1 - \frac{5\rho}{2\gamma} - L^2\gamma^2\right) + 40\gamma\rho}\Phi_0\\
        &=& \frac{\|x^0 - x^*\|^2 + 40\gamma\rho\|F(x^0)\|^2}{N\gamma^2\left(1 - \frac{5\rho}{2\gamma} - L^2\gamma^2\right) + 40\gamma\rho}\\
        &\overset{\eqref{eq:Lipschitzness}}{\leq}& \frac{(1 + 40\gamma\rho L^2)\|x^0 - x^*\|^2}{N\gamma^2\left(1 - \frac{5\rho}{2\gamma} - L^2\gamma^2\right) + 40\gamma\rho},
    \end{eqnarray*}
    which concludes the proof of \eqref{eq:EG_last_iter_apendix}. Moreover, \eqref{eq:EG_last_iterate} follows from \eqref{eq:EG_last_iter_apendix} since $4\rho \leq \gamma \leq \nicefrac{1}{2L}$ implies $1 - (\nicefrac{5\rho}{2\gamma}) - L^2\gamma^2 \geq \nicefrac{1}{8}$ and  $1 + 40\gamma\rho L^2 \leq \nicefrac{7}{2}$.
\end{proof}

\subsubsection{Counter-examples}

\begin{theorem}[Theorem~\ref{thm:EG_counter_example}]\label{thm:EG_counter_example_appendix}
    For any $L > 0$, $\rho \geq \nicefrac{1}{2L}$, and any choice of stepsizes $\gamma_1, \gamma_2 > 0$ there exists $\rho$-negative comonotone $L$-Lipschitz operator $F$ such that \ref{eq:EG} does not necessary converges on solving \ref{eq:VIP} with this operator $F$. In particular, for $\gamma_1 > \nicefrac{1}{L}$ it is sufficient to take $F(x) = L x$, and for $0 < \gamma_1 \leq \nicefrac{1}{L}$ one can take $F(x) = L A x$, where $x \in \R^2$, $$A = \begin{pmatrix} \cos\theta & -\sin\theta \\ \sin\theta & \cos\theta \end{pmatrix},\quad \theta = \frac{2\pi}{3}.$$ 
\end{theorem}
\begin{proof}
    Assume that $L > 0$ and $\rho \geq \nicefrac{1}{2L}$. We start with the case when $\gamma_1 > \nicefrac{1}{L}$. Consider operator $F(x) = Lx$. This operator is $L$-Lipschitz. Moreover, $F$ is monotone and, as the result, it is $\rho$-negative comonotone for any $\rho \geq 0$. The iterates produced by \ref{eq:EG} with $x^0 \neq 0$ satisfy
    \begin{equation*}
        \tx^k = (1 - L\gamma_1)x^k,\quad x^{k+1} = x^k - L\gamma_2\tx^k = (1 - L\gamma_2 + L^2\gamma_2\gamma_1)x^k\notag
    \end{equation*}
    implying that
    \begin{equation*}
        \|x^{k+1} - x^*\| = \|x^{k+1}\| = |1 - L\gamma_2 + L^2\gamma_2\gamma_1|\cdot \|x^k\| > \|x^k\| = \|x^k - x^*\|,
    \end{equation*}
    since $1 - L\gamma_2 + L^2\gamma_2\gamma_1 > 1 - L\gamma_2 + L\gamma_2 = 1$. That is, if $x^0 \neq 0$, then \ref{eq:EG} diverges in this case.
    
    Next, assume that $\gamma_1 < \nicefrac{1}{L}$ and consider $F(x) = L A x$, where $x \in \R^2$, $$A = \begin{pmatrix} \cos\theta & -\sin\theta \\ \sin\theta & \cos\theta \end{pmatrix},\quad \theta = \frac{2\pi}{3}.$$
    Operator $F$ is $L$-Lipschitz and $(\nicefrac{1}{2L})$-negative comonotone: for any $x,y \in \R^d$
    \begin{eqnarray*}
        \|F(x) - F(y)\| &=& L\|A(x-y)\| = L\|x-y\|,\\
        \langle F(x) - F(y), x-y \rangle &=& \|F(x) - F(y)\|\cdot \|x-y\| \cdot \cos \theta \\
        &=& \|F(x) - F(y)\|\cdot \|A(x-y)\| \cdot \cos \frac{2\pi}{3}\\
        &=& - \frac{1}{2L}\|F(x) - F(y)\|^2
    \end{eqnarray*}
    where we use the fact that $A$ is a rotation matrix. That is, $F(x)$ satisfies the conditions of the theorem. Taking into account that
    \begin{equation*}
        A = \begin{pmatrix} \cos\theta & -\sin\theta \\ \sin\theta & \cos\theta \end{pmatrix} = \begin{pmatrix} -\frac{1}{2} & -\frac{\sqrt{3}}{2} \\ \frac{\sqrt{3}}{2} & \frac{1}{2} \end{pmatrix},\quad A^2 = \begin{pmatrix} \cos(2\theta) & -\sin(2\theta) \\ \sin(2\theta) & \cos(2\theta) \end{pmatrix} = \begin{pmatrix} -\frac{1}{2} & \frac{\sqrt{3}}{2} \\ -\frac{\sqrt{3}}{2} & \frac{1}{2} \end{pmatrix},
    \end{equation*}
    we rewrite the update rule of \ref{eq:EG} as follows:
    \begin{eqnarray*}
        x^{k+1} &=& x^k - \gamma_2F\left(x^k - \gamma_2 F(x^k)\right)\\
        &=& x^k - \gamma_2 L A\left(x^k - \gamma_2 L A x^k\right)\\
        &=& \left(I - \gamma_2 L A + \gamma_1\gamma_2 L^2 A^2\right)x^k.
    \end{eqnarray*}
    To prove the divergence of \ref{eq:EG}, it remains to show that $\exists \lambda \in \Sp(I - \gamma_2 L A + \gamma_1\gamma_2 L^2 A^2)$ such that $|\lambda| > 1$. Indeed, we have
    \begin{eqnarray*}
        I - \gamma_2 L A + \gamma_1\gamma_2 L^2 A^2 &=& \begin{pmatrix}
            1 & 0 \\ 0 & 1
        \end{pmatrix} - \gamma_2 L \begin{pmatrix} -\frac{1}{2} & -\frac{\sqrt{3}}{2} \\ \frac{\sqrt{3}}{2} & \frac{1}{2} \end{pmatrix} + \gamma_1\gamma_2 L^2 \begin{pmatrix} -\frac{1}{2} & \frac{\sqrt{3}}{2} \\ -\frac{\sqrt{3}}{2} & \frac{1}{2} \end{pmatrix}\\
        &=& \begin{pmatrix}
            1 + \frac{\gamma_2 L}{2}(1 - \gamma_1 L) & \frac{\sqrt{3}\gamma_2L}{2}(1 + \gamma_1 L) \\ -\frac{\sqrt{3}\gamma_2L}{2}(1 + \gamma_1 L) & 1 + \frac{\gamma_2 L}{2}(1 - \gamma_1 L).
        \end{pmatrix}
    \end{eqnarray*}
    The above matrix has eigenvalues $\lambda_{1,2} = 1 + \frac{\gamma_2 L}{2}(1 - \gamma_1 L) \pm i \cdot \frac{\sqrt{3}\gamma_2L}{2}(1 + \gamma_1 L)$. Since $\gamma_2 > 0$ and $\gamma_1 \leq \nicefrac{1}{L}$ we have that $|\lambda_{1,2}|^2 = \left(1 + \frac{\gamma_2 L}{2}(1 - \gamma_1 L)\right)^2 + \frac{3}{4}\gamma_2^2L^2 (1 + \gamma_1 L)^2 > 1$. This concludes the proof.
\end{proof}

\subsection{Optimistic gradient}
\subsubsection{Guarantees for the averaged squared norm of the operator}

In this section, we proceed in an analogous way to the previous section on the Extragradient method. For notation convenience, we assume that $F(\tx^{-1}) = 0$. Then, the update rule for \ref{eq:OG} can be written as
\begin{equation}\tag{\algname{OG}}
    \begin{aligned}
    \tx^{k} &= x^k  - \gamma_1 F(\tx^{k-1}),\\
    x^{k+1} &= x^k - \gamma_2 F(\tx^k),
    \end{aligned} \quad \forall k \geq 0.
\end{equation}

\begin{lemma}\label{lem:OG_main_lemma_appendix}
    Let $F$ be $\rho$-star-negative comonotone. Then, for any $k\geq 0$ the iterates produced by \ref{eq:OG} after $k\geq 0$ iterations satisfy
    \begin{eqnarray}
        \|x^{k+1} - x^*\|^2 &\leq& \|x^k - x^*\|^2 - \gamma_2\left(\gamma_1 - 2\rho - \gamma_2\right)\|F(\tx^k)\|^2 - \gamma_1\gamma_2\|F(\tx^{k-1})\|^2\notag\\
        &&\quad + \gamma_1\gamma_2\|F(\tx^k) - F(\tx^{k-1})\|^2. \label{eq:OG_key_inequality}
    \end{eqnarray}
\end{lemma}
\begin{proof}
    By the definition of $x^{k+1}$ and $\tx^k$ we have
    \begin{eqnarray*}
        \|x^{k+1} - x^*\|^2 &=& \|x^k - x^*\|^2 - 2\gamma_2 \langle x^k - x^*, F(\tx^k) \rangle + \gamma_2^2 \|F(\tx^k)\|^2\\
        &=& \|x^k - x^*\|^2 - 2\gamma_2 \langle \tx^k - x^*, F(\tx^k) \rangle - 2\gamma_1\gamma_2 \langle F(\tx^{k-1}), F(\tx^k) \rangle + \gamma_2^2 \|F(\tx^k)\|^2.
    \end{eqnarray*}
    Next, we estimate the second term in the right-hand side using star-negative comonotonicity:
    \begin{eqnarray*}
        \|x^{k+1} - x^*\|^2 &\overset{\eqref{eq:rho_star_neg_comon}}{\leq}& \|x^k - x^*\|^2 + \gamma_2\left(2\rho + \gamma_2\right)\|F(\tx^k)\|^2 - 2\gamma_1\gamma_2 \langle F(\tx^{k-1}), F(\tx^k) \rangle.
    \end{eqnarray*}
    Finally, we handle the last term in the right-hand side of the above inequality using $2\langle a,b \rangle = \|a\|^2 + \|b\|^2 - \|a - b\|^2$, which holds for any $a,b \in \R^d$:
    \begin{eqnarray}
        \|x^{k+1} - x^*\|^2 &\leq& \|x^k - x^*\|^2 - \gamma_2\left(\gamma_1 - 2\rho - \gamma_2\right)\|F(\tx^k)\|^2 - \gamma_1\gamma_2\|F(\tx^{k-1})\|^2\notag\\
        &&\quad + \gamma_1\gamma_2\|F(\tx^{k-1}) - F(\tx^k)\|^2. \notag
    \end{eqnarray}
\end{proof}

This lemma is a building block of the first part of Theorem~\ref{thm:OG_convergence}.
\begin{theorem}[First part of Theorem~\ref{thm:OG_convergence}]
    Let $F$ be $L$-Lipschitz and $\rho$-star-negative comonotone with $\rho < \nicefrac{1}{2L}$. If $2\rho < \gamma_1 < \nicefrac{1}{L}$ and $0 < \gamma_2 < \min\{\nicefrac{1}{L} - \gamma_1, \gamma_1 - 2\rho\}$, then the iterates produced by \ref{eq:OG} after $N\geq 0$ iteration satisfy
    \begin{equation}
        \frac{1}{N+1}\sum\limits_{k=0}^N \|F(\tx^k)\|^2 \leq \frac{\|x^0 - x^*\|^2}{\gamma_1\gamma_2(1 - L^2(\gamma_1 + \gamma_2)^2)(N+1)}. \label{eq:OG_best_iterate_appendix}
    \end{equation}
\end{theorem}
\begin{proof}
    We first upper bound the last term that appeared in Lemma~\ref{lem:OG_main_lemma_appendix} using $L$-Lipschitzness:
    \begin{align*}
        \|F(\tx^k) - F(\tx^{k-1})\|^2 &\leq L^2 \| \tx^k - \tx^{k-1}\|^2 \\
        &= L^2 \| (\tx^k - x^k) + (x^k - x^{k-1})  + (x^{k-1} - \tx^{k-1})\|^2 \\
        &= L^2 \| (\gamma_1 + \gamma_2) F(\tx^{k-1})  - \gamma_1 F(\tx^{k-2})\|^2 \\ 
        &= L^2  (\gamma_1 + \gamma_2)^2 \| F(\tx^{k-1})\|^2  + L^2 \gamma_1^2  \|F(\tx^{k-2})\|^2 \\
        &\quad - 2L^2 (\gamma_1 + \gamma_2)\gamma_1 \langle  F(\tx^{k-1}),  F(\tx^{k-2}) \rangle.
    \end{align*}
    Decomposing the last term using $2\langle a,b \rangle =  \|a\|^2 + \|a\|^2 - \|a-b\|^2$ yields
    \begin{align}
        \|F(\tx^k) - F(\tx^{k-1})\|^2 &\leq L^2  (\gamma_1 + \gamma_2)\gamma_2 \| F(\tx^{k-1})\|^2  - L^2 \gamma_1 \gamma_2  \|F(\tx^{k-2})\|^2\notag\\
        &\quad+ L^2  \gamma_1 (\gamma_1 + \gamma_2) \|  F(\tx^{k-1}) - F(\tx^{k-2}) \|^2.\label{eq:OG_rec_1}
    \end{align}
    The above recurrence holds for $k \geq 2$. For $k=1$, we have $\tx^0 = x^0$ and $\tx^1 = x^1 - \gamma_1 F(\tx^0) = x^0 - (\gamma_1 + \gamma_2) F(x^0)$. Therefore, \begin{align}
    \label{eq:OG_rec_2}
        \|F(\tx^1) - F(\tx^{0})\|^2 &\leq L^2  (\gamma_1 + \gamma_2)^2 \| F(x^{0})\|^2.
    \end{align}
   Combining \eqref{eq:OG_rec_2} and \eqref{eq:OG_rec_1}, we obtain
    \begin{align*}
       \sum_{k=1}^N \|F(\tx^k) - F(\tx^{k-1})\|^2 &\leq \sum_{k=2}^N \left(L^2  (\gamma_1 + \gamma_2)\gamma_2 \| F(\tx^{k-1})\|^2  - L^2 \gamma_1 \gamma_2  \|F(\tx^{k-2})\|^2\right) \\
       &\quad + \sum_{k=2}^N L^2  \gamma_1 (\gamma_1 + \gamma_2) \|  F(\tx^{k-1}) - F(\tx^{k-2}) \|^2 + L^2  (\gamma_1 + \gamma_2)^2 \| F(x^{0})\|^2.
    \end{align*}
    We can simplify the terms on the right-hand side. Therefore, 
    \begin{align*}
       \sum_{k=1}^N \|F(\tx^k) - F(\tx^{k-1})\|^2 &\leq L^2  (\gamma_1 + \gamma_2)\gamma_2 \| F(\tx^{N-1})\|^2 + \sum_{k=2}^N L^2  \gamma_2^2 \| F(\tx^{k-1})\|^2  \\
       &\quad + \sum_{k=2}^N L^2  \gamma_1 (\gamma_1 + \gamma_2) \|  F(\tx^{k-1}) - F(\tx^{k-2}) \|^2 \\
       &\quad + L^2  (\gamma_1^2 + \gamma_1\gamma_2 + \gamma_2^2) \| F(x^{0})\|^2 \\
       &\leq \sum_{k=2}^N L^2  (\gamma_1 + \gamma_2)\gamma_2 \| F(\tx^{k-1})\|^2 \\
       &\quad + \sum_{k=1}^{N} L^2  \gamma_1 (\gamma_1 + \gamma_2) \|  F(\tx^{k}) - F(\tx^{k-1}) \|^2 \\
       &\quad  + L^2  (\gamma_1^2 + \gamma_1\gamma_2 + \gamma_2^2) \| F(x^{0})\|^2.
    \end{align*}
    Using the above equation, we can bound $\sum_{k=1}^N \|F(\tx^k) - F(\tx^{k-1})\|^2$, i.e.,
    \begin{align}
       (1 - L^2  \gamma_1 (\gamma_1 + \gamma_2))\sum_{k=1}^N \|F(\tx^k) - F(\tx^{k-1})\|^2 &\leq \sum_{k=1}^N L^2  (\gamma_1 + \gamma_2)\gamma_2 \| F(\tx^{k-1})\|^2 \notag \\
       &\quad + L^2 \gamma_1^2 \| F(x^{0})\|^2.  \label{eq:OG_rec_diff}
    \end{align}
    Let us now apply \eqref{eq:OG_key_inequality} recursively ($F(\tx^{-1}) = 0$ for simpler notation), which leads to
    \begin{align*}
        \|x^{N} - x^*\|^2 &\leq \|x^0 - x^*\|^2 - \sum_{k=0}^{N-1}\left(\gamma_2\left(\gamma_1 - 2\rho - \gamma_2\right)\|F(\tx^k)\|^2 + \gamma_1\gamma_2\|F(\tx^{k-1})\|^2\right) \\
        &\quad + \sum_{k=0}^{N-1} \gamma_1\gamma_2\|F(\tx^k) - F(\tx^{k-1})\|^2)\\
        &\leq \|x^0 - x^*\|^2 - \sum_{k=0}^{N-2}\gamma_2\left(2\gamma_1 - 2\rho - \gamma_2\right)\|F(\tx^k)\|^2  + \sum_{k=1}^{N-1}\gamma_1\gamma_2\|F(\tx^k) - F(\tx^{k-1})\|^2 \\
        &\quad + \gamma_1\gamma_2 \|F(x^0)\|^2. 
    \end{align*}
    Plugging \eqref{eq:OG_rec_diff} to the above with $1 - L^2  \gamma_1 (\gamma_1 + \gamma_2) > 0$, which follows from $\gamma_1 < \nicefrac{1}{L}$ and $\gamma_2 < \nicefrac{1}{L} - \gamma_1$, leads to
    \begin{align*}
        \|x^{N} - x^*\|^2 &\leq \|x^0 - x^*\|^2 - \sum_{k=0}^{N-2}\gamma_2\left(2\gamma_1 - 2\rho - \gamma_2\right)\|F(\tx^k)\|^2 + \gamma_1\gamma_2 \|F(x^0)\|^2 \\
        &\quad + \frac{\gamma_1\gamma_2^2 (\gamma_1 + \gamma_2) L^2}{1 - L^2  \gamma_1 (\gamma_1 + \gamma_2)}\sum_{k=0}^{N-2} \| F(\tx^{k})\|^2 + \frac{\gamma_1^3\gamma_2 L^2 }{1 - L^2  \gamma_1 (\gamma_1 + \gamma_2)} \| F(x^{0})\|^2. 
    \end{align*}
    Since $\gamma_1 - 2\rho - \gamma_2 > 0$, we have
     \begin{align*}
        \|x^{N} - x^*\|^2 &\leq \|x^0 - x^*\|^2 -  \gamma_1\gamma_2\sum_{k=0}^{N-2}\|F(\tx^k)\|^2 + \gamma_1\gamma_2 \|F(x^0)\|^2 \\
        &\quad + \frac{\gamma_1\gamma_2^2 (\gamma_1 + \gamma_2) L^2}{1 - L^2  \gamma_1 (\gamma_1 + \gamma_2)}\sum_{k=0}^{N-2} \| F(\tx^{k})\|^2 + \frac{\gamma_1^3\gamma_2 L^2 }{1 - L^2  \gamma_1 (\gamma_1 + \gamma_2)} \| F(x^{0})\|^2. 
    \end{align*}
    Rearranging terms and applying $\|x^{N} - x^*\|^2 \geq 0$ plus $\|F(x^0)\|^2 \leq L^2\|x^0 - x^*\|^2$, we obtain
    \begin{align*}
        \gamma_1\gamma_2(1 - L^2(\gamma_1 + \gamma_2)^2)\sum_{k=0}^{N-2}\|F(\tx^k)\|^2 &\leq (1 - L^2  \gamma_1 (\gamma_1 + \gamma_2))\|x^0 - x^*\|^2 \\
        &\quad + \gamma_1\gamma_2 L^2 (1 - L^2\gamma_1\gamma_2) \|x^0-x^*\|^2 \\
        &\leq (1-L^2\gamma_1^2)\|x^0-x^*\|^2 \leq \|x^0-x^*\|^2. 
    \end{align*}
    The above inequality implies \eqref{eq:OG_best_iterate_appendix}, since one can replace $N$ with $N+2$.
\end{proof}

\subsubsection{Last-iterate guarantees}

Our proof is based on the following lemma.

\begin{lemma}\label{lem:potential_OG}
    Let $F$ be $L$-Lipschitz and $\rho$-negative comonotone such that $\rho \leq \nicefrac{5}{62L}$. Then for any $k \geq 1$ the iterates produced by \ref{eq:EG} with $\gamma_1 = \gamma_2 = \gamma > 0$ such that $4\rho \leq \gamma \leq \nicefrac{10}{31L}$ satisfy
    \begin{eqnarray}
        \|F(x^{k+1})\|^2 + \|F(x^{k+1}) - F(\tx^k)\|^2  &\leq& \|F(x^{k})\|^2 + \|F(x^{k}) - F(\tx^{k-1})\|^2\notag\\
        &&\quad - \frac{1}{100}\|F(\tx^k) - F(\tx^{k-1})\|^2. \label{eq:OG_potential}
    \end{eqnarray}
\end{lemma}
\begin{proof}
    From $\rho$-negative comonotonicity and $L$-Lipschitzness of $F$ we have
    \begin{eqnarray*}
        -\rho\|F(x^{k+1}) - F(\tx^{k})\|^2 &\leq& \langle F(x^{k+1}) - F(\tx^k), x^{k+1} - \tx^k \rangle,\\
        -\rho\|F(x^{k}) - F(x^{k+1})\|^2 &\leq& \langle F(x^{k}) - F(x^{k+1}), x^{k} - x^{k+1} \rangle,\\
        \|F(x^{k+1}) - F(\tx^k)\|^2 &\leq& L^2\|x^{k+1} - \tx^k\|^2.
    \end{eqnarray*}
    Using $x^{k+1} - \tx^k = \gamma (F(\tx^{k-1}) - F(\tx^{k}))$ and $x^{k+1} - x^k = - \gamma F(\tx^k)$, we rewrite the above inequalities as
    \begin{eqnarray}
        -\frac{\rho}{\gamma}\|F(x^{k+1}) - F(\tx^{k})\|^2 &\leq& \langle F(x^{k+1}) - F(\tx^k), F(\tx^{k-1}) - F(\tx^{k}) \rangle, \label{eq:OG_technical_1}\\
        -\frac{\rho}{\gamma}\|F(x^{k}) - F(x^{k+1})\|^2 &\leq& \langle F(x^{k}) - F(x^{k+1}), F(\tx^k) \rangle, \label{eq:OG_technical_2}\\
        \|F(x^{k+1}) - F(\tx^k)\|^2 &\leq& L^2\gamma^2\|F(\tx^{k-1}) - F(\tx^{k})\|^2. \label{eq:OG_technical_3}
    \end{eqnarray}
    Next, we apply $2\langle a,b \rangle = \|a\|^2 + \|b\|^2 - \|a - b\|^2$, which holds for any $a,b \in \R^d$, and from the first two inequalities get
    \begin{eqnarray}
        -\frac{2\rho}{3\gamma}\|F(x^{k+1}) - F(\tx^{k})\|^2 &\overset{\eqref{eq:OG_technical_1}}{\leq}& \frac{1}{3}\|F(x^{k+1}) - F(\tx^k)\|^2 + \frac{1}{3}\|F(\tx^{k-1}) - F(\tx^k)\|^2\notag\\
        &&\quad - \frac{1}{3}\|F(x^{k+1}) - F(\tx^{k-1})\|^2, \label{eq:OG_technical_1_1}\\
        -\frac{2\rho}{\gamma}\|F(x^{k}) - F(x^{k+1})\|^2 &\overset{\eqref{eq:OG_technical_2}}{\leq}& 2\langle F(x^{k}), F(\tx^k) \rangle - 2\langle F(x^{k+1}), F(\tx^k) \rangle\notag\\
        &=& \|F(x^k)\|^2 + \|F(\tx^k)\|^2 - \|F(x^k) - F(\tx^k)\|^2 \notag\\
        &&\quad - \|F(x^{k+1})\|^2 - \|F(\tx^k)\|^2 + \|F(x^{k+1}) - F(\tx^k)\|^2 \notag\\
        &=& \|F(x^k)\|^2 + \|F(x^{k+1}) - F(\tx^k)\|^2 - \|F(x^{k+1})\|^2\notag\\
        &&\quad - \|F(x^k) - F(\tx^k)\|^2.\label{eq:OG_technical_2_1}
    \end{eqnarray}
    Summing up \eqref{eq:OG_technical_3}, \eqref{eq:OG_technical_1_1}, and \eqref{eq:OG_technical_2_1} with weights $3$, $1$, and $1$ respectively, we derive
    \begin{eqnarray*}
        \left(3 - \frac{2\rho}{3\gamma}\right)\|F(x^{k+1}) - F(\tx^{k})\|^2 - \frac{2\rho}{\gamma}\|F(x^{k}) - F(x^{k+1})\|^2 &\leq& 3L^2\gamma^2\|F(\tx^{k-1}) - F(\tx^{k})\|^2\\
        &&\quad + \frac{1}{3}\|F(x^{k+1}) - F(\tx^k)\|^2\\
        &&\quad + \frac{1}{3}\|F(\tx^{k-1}) - F(\tx^k)\|^2\\
        &&\quad - \frac{1}{3}\|F(x^{k+1}) - F(\tx^{k-1})\|^2\\
        &&\quad + \|F(x^k)\|^2 + \|F(x^{k+1})- F(\tx^k)\|^2\\
        &&\quad- \|F(x^{k+1})\|^2 - \|F(x^k) - F(\tx^k)\|^2\\
        &=& \left(\frac{1}{3} + 3L^2\gamma^2\right)\|F(\tx^{k-1}) - F(\tx^{k})\|^2\\
        &&\quad + \frac{4}{3}\|F(x^{k+1})- F(\tx^k)\|^2\\
        &&\quad - \frac{1}{3}\|F(x^{k+1}) - F(\tx^{k-1})\|^2\\
        &&\quad + \|F(x^k)\|^2 - \|F(x^{k+1})\|^2\\
        &&\quad - \|F(x^k) - F(\tx^k)\|^2.
    \end{eqnarray*}
    To simplify further derivations, we introduce new notation: $\Psi_k = \|F(x^{k})\|^2 + \|F(x^{k}) - F(\tx^{k-1})\|^2$, $\forall k \geq 1$. Rearranging the terms in the above inequality and using the new notation, we arrive at
    \begin{eqnarray*}
        \Psi_{k+1} - \Psi_k &\leq& T_k,\quad \text{where}\\
        T_k &\eqdef& \frac{2\rho}{\gamma}\|F(x^{k}) - F(x^{k+1})\|^2 + \left(\frac{1}{3} + 3L^2\gamma^2\right)\|F(\tx^{k-1}) - F(\tx^{k})\|^2\\
        &&\quad - \frac{2}{3}\left(1 - \frac{\rho}{\gamma}\right)\|F(x^{k+1}) - F(\tx^k)\|^2 - \|F(x^k) - F(\tx^{k-1})\|^2\\
        &&\quad - \frac{1}{3}\|F(x^{k+1}) - F(\tx^{k-1})\|^2 - \|F(x^{k}) - F(\tx^k)\|^2.
    \end{eqnarray*}
    To prove \eqref{eq:OG_potential}, it remains to show that $T_k \leq - \frac{1}{100}\|F(\tx^k) - F(\tx^{k-1})\|^2$ for all $k \geq 1$. Taking into account $4\rho \leq \gamma \leq \nicefrac{10}{31L}$, we upper bound $T_k$ as follows:
    \begin{eqnarray}
        T_k &\leq& \frac{1}{2}\|F(x^{k}) - F(x^{k+1})\|^2 + \frac{121}{93}\|F(\tx^{k-1}) - F(\tx^{k})\|^2 - \frac{1}{2}\|F(x^{k+1}) - F(\tx^k)\|^2\notag\\
        &&\quad - \|F(x^k) - F(\tx^{k-1})\|^2 - \frac{1}{3}\|F(x^{k+1}) - F(\tx^{k-1})\|^2 - \|F(x^{k}) - F(\tx^k)\|^2\notag\\
        &=& \begin{pmatrix} F(x^{k+1}) \\ F(x^k) \\  F(\tx^{k}) \\  F(\tx^{k-1})\end{pmatrix}^\top \left(\begin{pmatrix} -\frac{1}{3} & -\frac{1}{2} &  \frac{1}{2} & \frac{1}{3} \\ -\frac{1}{2} & -\frac{3}{2} & 1 & 1\\ \frac{1}{2} & 1 & -\frac{4927}{5766} & - \frac{1861}{2883}\\ \frac{1}{3} & 1 & - \frac{1861}{2883} & - \frac{661}{961}\end{pmatrix} \otimes  I_d\right) \begin{pmatrix} F(x^{k+1}) \\ F(x^k) \\  F(\tx^{k}) \\  F(\tx^{k-1})\end{pmatrix}, \label{eq:OG_technical_4}
    \end{eqnarray}
    where $I_d$ is $d$-dimensional identity matrix and $A\otimes B$ denotes the Kronecker product of two matrices $A$ and $B$. One can show numerically (see our \href{https://github.com/eduardgorbunov/Proximal_Point_and_Extragradient_based_methods_negative_comonotonicity/blob/main/Optimistic_Gradient/Lemma_C8_symbolic_verification_and_matrix_estimation.ipynb}{codes}) that 
    \begin{equation*}
        \begin{pmatrix} -\frac{1}{3} & -\frac{1}{2} &  \frac{1}{2} & \frac{1}{3} \\ -\frac{1}{2} & -\frac{3}{2} & 1 & 1\\ \frac{1}{2} & 1 & -\frac{4927}{5766} & - \frac{1861}{2883}\\ \frac{1}{3} & 1 & - \frac{1861}{2883} & - \frac{661}{961}\end{pmatrix} \preccurlyeq -\frac{1}{100}\begin{pmatrix} 0 & 0 &  0 & 0 \\ 0 & 0 &  0 & 0\\ 0 & 0 &  1 & -1\\ 0 & 0 &  -1 & 1\end{pmatrix}.
    \end{equation*}
    Therefore, in view of \eqref{eq:OG_technical_4}, we have 
    \begin{eqnarray*}
        T_K &\leq& - \frac{1}{100} \begin{pmatrix} F(x^{k+1}) \\ F(x^k) \\  F(\tx^{k}) \\  F(\tx^{k-1})\end{pmatrix}^\top \left(\begin{pmatrix} 0 & 0 &  0 & 0 \\ 0 & 0 &  0 & 0\\ 0 & 0 &  1 & -1\\ 0 & 0 &  -1 & 1\end{pmatrix} \otimes  I_d\right) \begin{pmatrix} F(x^{k+1}) \\ F(x^k) \\  F(\tx^{k}) \\  F(\tx^{k-1})\end{pmatrix}\\
        &=& - \frac{1}{100}\|F(\tx^k) - F(\tx^{k-1})\|^2,
    \end{eqnarray*}
    which concludes the proof.
\end{proof}

We emphasize that similar potential $\|F(x^{k})\|^2 + \|F(x^{k}) - F(\tx^{k-1})\|^2$ is used in \citet{cai2022tight} to prove the last-iterate convergence of \ref{eq:OG} for \emph{monotone} $L$-Lipschitz $F$. However, our proof non-trivially differs from the one from \citet{cai2022tight}: we use $\rho$-negative comonotonicity for pairs $(x^{k+1}, \tx^k)$, $(x^{k+1}, x^k)$ and $L$-Lipschitzness for $(x^{k+1}, \tx^k)$, while \citet{cai2022tight} use \emph{monotonicity} $(x^{k+1}, x^k)$ and $L$-Lipschitzness for $(x^{k+1}, \tx^k)$. Next, the only known last-iterate $\cO(\nicefrac{1}{N})$ convergence rate for \ref{eq:OG} under $\rho$-negative comonotonicity and $L$-Lipschitzness \citep{luo2022last} is based on a different potential $\|F(x^{k+1})\|^2 + \frac{2\gamma - 3\rho}{2\gamma}\|F(x^{k+1}) - F(\tx^k)\|^2$, which coincides with the one used by \citet{gorbunov2022last} in the monotone case. Moreover, the proof from \citet{luo2022last} is based on $2$ inequalities: $\rho$-negative comonotonicity for pairs $(x^{k+1}, x^k)$ and $L$-Lipschitzness for $(x^{k+1}, \tx^k)$. In contrast, our proof is based on $3$ inequalities, i.e., it uses more information about the problem. This might be the reason, why our proof allows $\rho$ to be larger than in the proof by \citet{luo2022last}.

Using Lemma~\ref{lem:potential_OG}, we can proceed with the potential-based proof of the last-iterate convergence of \ref{eq:OG}.
\begin{theorem}[Second part of Theorem~\ref{thm:OG_convergence}]
    Let $F$ be $L$-Lipschitz and $\rho$-negative comonotone. Then, for any $k \geq 0$ the iterates produced by \ref{eq:OG} with $\gamma_1 = \gamma_2 = \gamma$ such that $4\rho \leq \gamma \leq \nicefrac{10}{31L}$ satisfy
    \begin{equation}
        \Phi_{k+1} \leq \Phi_k,\quad \text{where}\quad \Phi_k = \|x^k - x^*\|^2 + \left(k \frac{\gamma(\gamma - 3\rho)}{2  + 6L^2\gamma^2} + 400\gamma^2\right)\Psi_k, \label{eq:OG_potential_full}
    \end{equation}
    where $\Psi_k = \|F(x^{k})\|^2 + \|F(x^{k}) - F(\tx^{k-1})\|^2$.
    That is, under the introduced assumptions on $\gamma$ and $\rho$ for any $N \geq 1$ the iterates produced by \ref{eq:OG} satisfy
        \begin{equation}
        \|F(x^N)\|^2 \leq \frac{717\|x^0 - x^*\|^2}{N \gamma(\gamma - 3\rho) + 800\gamma^2}. \label{eq:OG_last_iter_apendix}
    \end{equation}
\end{theorem}
\begin{proof}
    From \eqref{eq:OG_key_inequality} with $\gamma_1 = \gamma_2 = \gamma$ we have
    \begin{equation*}
      \|x^{k+1} - x^*\|^2 \leq \|x^k - x^*\|^2 + 2\rho\gamma\|F(\tx^k)\|^2 - \gamma^2\|F(\tx^{k-1})\|^2  + \gamma^2\|F(\tx^k) - F(\tx^{k-1})\|^2.
    \end{equation*}
    Next, taking into account that $4\rho \leq \gamma \leq \nicefrac{10}{31L}$, we also have from Lemma~\ref{lem:potential_OG} the following inequality:
    \begin{equation*}
        \|F(x^{k+1})\|^2 + \|F(x^{k+1}) - F(\tx^k)\|^2  \leq \|F(x^{k})\|^2 + \|F(x^{k}) - F(\tx^{k-1})\|^2  - \frac{1}{100}\|F(\tx^k) - F(\tx^{k-1})\|^2.
    \end{equation*}
    Using these two inequalities, we derive the following upper bound on $\Phi_{k+1}$:
    \begin{align*}
        \Phi_{k+1} &= \|x^{k+1} - x^*\|^2 + \left((k+1) \frac{\gamma(\gamma - 3\rho)}{2  + 6L^2\gamma^2} + 400\gamma^2\right)\Psi_{k+1} \\
        &\leq \|x^k - x^*\|^2 + 2\rho\gamma\|F(\tx^k)\|^2 - \gamma^2\|F(\tx^{k-1})\|^2  + \gamma^2\|F(\tx^k) - F(\tx^{k-1})\|^2 \\
        &\quad + \left((k+1) \frac{\gamma(\gamma - 3\rho)}{2  + 6L^2\gamma^2} + 400\gamma^2\right) \left( \|F(x^{k})\|^2 + \|F(x^{k}) - F(\tx^{k-1})\|^2  - \frac{1}{100}\|F(\tx^k) - F(\tx^{k-1})\|^2 \right) \\
        &\leq \Phi_{k} + 2\rho\gamma\|F(\tx^k)\|^2 - \gamma^2\|F(\tx^{k-1})\|^2  + \gamma^2\|F(\tx^k) - F(\tx^{k-1})\|^2 \\
         &\quad + \frac{\gamma(\gamma - 3\rho)}{2  + 6L^2\gamma^2} \left(\|F(x^{k})\|^2 + \|F(x^{k}) - F(\tx^{k-1})\|^2\right) - 4\gamma^2\|F(\tx^k) - F(\tx^{k-1})\|^2  \\
         &= \Phi_{k} + 2\rho\gamma\|F(\tx^k)\|^2 - \gamma^2\|F(\tx^{k-1})\|^2 \\
         &\quad + \frac{\gamma(\gamma - 3\rho)}{2  + 6L^2\gamma^2} \left(\|F(x^{k})\|^2 + \|F(x^{k}) - F(\tx^{k-1})\|^2\right) - 3\gamma^2\|F(\tx^k) - F(\tx^{k-1})\|^2.
    \end{align*}
    In the next step, we use following upper bounds based on $L$-Lipschitzness, \eqref{eq:OG}, and $\|a+b\|^2 \leq (1+\beta)\|a\|^2 + (1+\beta^{-1})\|b\|^2$, which holds $\forall a,b\in \R^d$, $\beta > 0$:
    \begin{align*}
         \|F(\tx^k)\|^2 &\leq 3\|F(\tx^k) - F(\tx^{k-1})\|^2 + \nicefrac{3}{2}\|F(\tx^{k-1})\|^2, \\
         \|F(x^k) - F(\tx^{k-1})\|^2 &\leq 2\|F(x^k) - F(\tx^k)\|^2 + 2\|F(\tx^k) - F(\tx^{k-1})\|^2 \\
         &\leq 2L^2\gamma^2\|F(\tx^{k-1})\|^2 + 2\|F(\tx^k) - F(\tx^{k-1})\|^2, \\ 
        \|F(x^k)\|^2 &\leq 2\|F(x^k) - F(\tx^{k-1})\|^2 + 2\|F(\tx^{k-1})\|^2\\
        &\leq 2(1 + 2L^2\gamma^2)\|F(\tx^{k-1}) \|^2 + 4\|F(\tx^k) - F(\tx^{k-1})\|^2.
    \end{align*}
    Applying the above yields
    \begin{align*}
        \Phi_{k+1} &\leq \Phi_{k} + \left(\frac{\gamma(\gamma - 3\rho)}{2  + 6L^2\gamma^2}(2  + 6L^2\gamma^2) + 3\rho\gamma - \gamma^2\right)\|F(\tx^{k-1})\|^2\\
        &\quad + \left(6\rho\gamma + \frac{3\gamma(\gamma - 3\rho)}{1+3L^2\gamma^2} - 3\gamma^2\right)\|F(\tx^k) - F(\tx^{k-1})\|^2\\
        &\leq \Phi_{k},
    \end{align*}
    where we use $1 + 3L^2\gamma^2 \geq 1$. This applies that $\Phi_{N} \leq \Phi_1$, therefore
    \begin{align*}
        \|F(x^N)\|^2 &\leq \frac{\|x^1 - x^*\|^2 + \left(\frac{\gamma(\gamma - 3\rho)}{2  + 6L^2\gamma^2} + 400\gamma^2\right)\left(\|F(x^{1})\|^2 + \|F(x^{1}) - F(x^0)\|^2\right)}{N \frac{\gamma(\gamma - 3\rho)}{2  + 6L^2\gamma^2} + 400\gamma^2}
    \end{align*}
    In the next step, we bound everything with respect to $\|x^0 - x^*\|^2$ using $\rho$-negative comonotonicity, $L$-Lipschitzness, \eqref{eq:OG}, and $\|a+b\|^2 \leq (1+\beta)\|a\|^2 + (1+\beta^{-1})\|b\|^2$: 
    \begin{align*}
       \|x^1 - x^*\|^2 &\overset{\eqref{eq:OG_key_inequality}}{\leq} \|x^0 - x^*\|^2 + \gamma (2\rho + \gamma)\|F(x^0)\|^2 \\
       &\leq (1 + L^2\gamma (2\rho + \gamma))\|x^0 - x^*\|^2 \leq 2\|x^0 - x^*\|^2, \\
       \|F(x^{1})\|^2 + \|F(x^{1}) - F(x^0)\|^2 &\leq 3\|F(x^1)\|^2 + 2\|F(x^{0})\|^2 \\
       &\leq L^2 (3\|x^1 - x^*\|^2 + 2\|x^0 - x^*\|^2) \leq 8L^2\|x^0 - x^*\|^2, \\
       2 &\leq 2  + 6L^2\gamma^2 \leq 3.
    \end{align*}
    Putting all together yields
    \begin{align*}
        \|F(x^N)\|^2 &\leq \frac{(2 + 401 L^2 \cdot 8\gamma^2)\|x^0 - x^*\|^2}{N \frac{\gamma(\gamma - 3\rho)}{2} + 400\gamma^2} 
        \leq \frac{717\|x^0 - x^*\|^2}{N \gamma(\gamma - 3\rho) + 800\gamma^2},
    \end{align*}
    which concludes the proof.
\end{proof}

\subsubsection{Counter-examples}
\begin{theorem}[Theorem~\ref{thm:OG_counter_example}]\label{thm:OG_counter_example_appendix}
    For any $L > 0$, $\rho \geq \nicefrac{1}{2L}$, and any choice of stepsizes $\gamma_1, \gamma_2 > 0$ there exists $\rho$-negative comonotone $L$-Lipschitz operator $F$ such that \ref{eq:OG} does not necessary converges on solving \ref{eq:VIP} with this operator $F$. In particular, for $\gamma_1 > \nicefrac{1}{L}$ it is sufficient to take $F(x) = L x$, where $x \in \R$, and for $0 < \gamma_1 \leq \nicefrac{1}{L}$ one can take $F(x) = L A x$, where $x \in \R^2$, $$A = \begin{pmatrix} \cos\theta & -\sin\theta \\ \sin\theta & \cos\theta \end{pmatrix},\quad \theta = \frac{2\pi}{3}.$$ 
\end{theorem}
\begin{proof}
    Assume that $L > 0$ and $\rho \geq \nicefrac{1}{2L}$. We start with the case when $\gamma_1 > \nicefrac{1}{L}$. Consider operator $F(x) = Lx$. This operator is $L$-Lipschitz. Moreover, $F$ is monotone and, as the result, it is $\rho$-negative comonotone for any $\rho \geq 0$. The iterates produced by \ref{eq:OG} with $x^0 \neq 0$ satisfy
    \begin{equation*}
        \tx^k = x^k - L\gamma_1 \tx^{k-1},\quad x^{k+1} = x^k - L\gamma_2\tx^k = (1 - L\gamma_2)x^k + L^2\gamma_2\gamma_1\tx^{k-1}\notag
    \end{equation*}
    implying that
    \begin{equation*}
       \begin{bmatrix} 
       x^{k+1} \\
       \tx^k
       \end{bmatrix} = 
       \begin{pmatrix}
       1 - \gamma_2 L & \gamma_1\gamma_2L^2\\ 
       1 & - \gamma_1 L
       \end{pmatrix}
       \begin{bmatrix} 
       x^k \\
       \tx^{k-1}
       \end{bmatrix}
    \end{equation*}
    The eigenvalues of the above $2 \times 2$ matrix can be computed analytically. One of them has the form 
    \begin{align*}
        -\frac{L\gamma_1 + L\gamma_2 + \sqrt{L^2\gamma_1^2 + L^2\gamma_2^2 + 2L^2\gamma_1\gamma_2 + 2L\gamma_1 - 2L\gamma_2 + 1} - 1}{2} &\\
        &\hspace{-3cm}< -\frac{1 + \sqrt{1 + 2L\gamma_2 + 2 - 2L\gamma_2 + 1} - 1}{2} = -1.
    \end{align*}
    The derivation above is verified symbolically in our \href{https://github.com/eduardgorbunov/Proximal_Point_and_Extragradient_based_methods_negative_comonotonicity/blob/main/Optimistic_Gradient/Symbolic_verification_for_Theorem_C10.ipynb}{codes}. That means we can select such starting setup $x^0$, $\tx^0$, such that \ref{eq:OG} diverges.
    
    Next, assume that $\gamma_1 \leq \nicefrac{1}{L}$ and consider $F(x) = L A x$, where $x \in \R^2$, $$A = \begin{pmatrix} \cos\theta & -\sin\theta \\ \sin\theta & \cos\theta \end{pmatrix},\quad \theta = \frac{2\pi}{3}.$$
    Operator $F$ is $L$-Lipschitz and $(\nicefrac{1}{2L})$-negative comonotone: for any $x,y \in \R^d$
    \begin{eqnarray*}
        \|F(x) - F(y)\| &=& L\|A(x-y)\| = L\|x-y\|,\\
        \langle F(x) - F(y), x-y \rangle &=& \|F(x) - F(y)\|\cdot \|x-y\| \cdot \cos \theta \\
        &=& \|F(x) - F(y)\|\cdot \|A(x-y)\| \cdot \cos \frac{2\pi}{3}\\
        &=& - \frac{1}{2L}\|F(x) - F(y)\|^2
    \end{eqnarray*}
    where we use the fact that $A$ is a rotation matrix. That is, $F(x)$ satisfies the conditions of the theorem. Taking into account that
    \begin{equation*}
        A = \begin{pmatrix} \cos\theta & -\sin\theta \\ \sin\theta & \cos\theta \end{pmatrix} = \begin{pmatrix} -\frac{1}{2} & -\frac{\sqrt{3}}{2} \\ \frac{\sqrt{3}}{2} & \frac{1}{2} \end{pmatrix},\quad A^2 = \begin{pmatrix} \cos(2\theta) & -\sin(2\theta) \\ \sin(2\theta) & \cos(2\theta) \end{pmatrix} = \begin{pmatrix} -\frac{1}{2} & \frac{\sqrt{3}}{2} \\ -\frac{\sqrt{3}}{2} & \frac{1}{2} \end{pmatrix},
    \end{equation*}
    we rewrite the update rule of \ref{eq:OG} similarly to the case $\gamma_1 > \nicefrac{1}{L}:$
    \begin{equation*}
       \begin{bmatrix} 
       x^{k+1} \\
       \tx^k
       \end{bmatrix} = 
       \begin{pmatrix}
       I - \gamma_2 LA & \gamma_1\gamma_2L^2 A^2\\ 
       I & - \gamma_1 LA
       \end{pmatrix}
       \begin{bmatrix} 
       x^k \\
       \tx^{k-1}
       \end{bmatrix}
       = B
       \begin{bmatrix} 
       x^k \\
       \tx^{k-1}
       \end{bmatrix}.
    \end{equation*}
    To prove the divergence of \ref{eq:OG}, we show that $B$ is expansive, i.e., its spectral norm $\|B\| > 1$, therefore, there exists a starting point for which \ref{eq:OG} does not converge. The spectral norm of $B$ has the following form
    \begin{align*}
        \|B\|^2 = c + \sqrt{c^2 - L^2\gamma_1^2},
    \end{align*}
    where $c = \frac{L^4\gamma_1^2\gamma_2^2 + L^2\gamma_1^2 + L^2\gamma_2^2 + L\gamma_2}{2} + 1$ (this derivation is verified symbolically in our \href{https://github.com/eduardgorbunov/Proximal_Point_and_Extragradient_based_methods_negative_comonotonicity/blob/main/Optimistic_Gradient/Symbolic_verification_for_Theorem_C10.ipynb}{codes}). Therefore, $\|B\|$ is well defined since $c > 1$ and  $L^2\gamma_1^2 \leq 1$, and, moreover, $\|B\| > 1$, which concludes the proof.  
\end{proof}

\end{document}